\documentclass[12pt]{article}

\usepackage{epsfig}

\usepackage{caption}
\usepackage{amsfonts}

\usepackage{amssymb}
\usepackage{mathrsfs}
\usepackage{amsmath} 

\usepackage{graphicx} 
\usepackage{float}
\usepackage{pict2e}
\usepackage{enumerate}
\usepackage{enumitem}
\usepackage{authblk} 

\usepackage{graphics} 
\usepackage{tikz} 
\usepackage{ulem} 
\usepackage{cases}
\usetikzlibrary{calc}

\usepackage[f]{esvect}

\usepackage{tikz-3dplot}
\usetikzlibrary{hobby}

\usepackage{subcaption} 
\newcounter{FC}

\usetikzlibrary{shapes.geometric}
\usetikzlibrary{decorations.pathmorphing,calligraphy}

\numberwithin{equation}{section}

\usepackage[pdfauthor={derajan},pdftitle={How to do this},pdfstartview=XYZ,bookmarks=true,
colorlinks=true,linkcolor=blue,urlcolor=blue,citecolor=blue,pdftex,bookmarks=true,linktocpage=true,hyperindex=true]{hyperref}

\usepackage{subdepth}

\usepackage{color}

\definecolor{darkgreen}{RGB}{0,154,23}
\definecolor{darkbrown}{RGB}{183,60,18}

\def \re #1{#1}
\def \bl #1{#1}

\newcommand{\F}{\mathcal{F}}

\usepackage[colorinlistoftodos,prependcaption,textsize=scriptsize, textwidth=24mm,]{todonotes}

\usetikzlibrary{patterns}

\makeatother

\newtheorem{theorem}{Theorem}
\newtheorem{observation}[theorem]{Observation}

\newtheorem{definition}[theorem]{Definition}

\newtheorem{question}[theorem]{Question}
\newtheorem{lemma}[theorem]{Lemma}

\newtheorem{remark}{Remark}

\newtheorem{claim}{Claim}

\makeatletter
\newcommand\figcaption{\def\@captype{figure}\caption}
\newcommand\tabcaption{\def\@captype{table}\caption}
\makeatother

\newtheorem{conjecture}[theorem]{Conjecture}

\newenvironment{proof}{\noindent {\bf
		Proof.}}{\rule{2.5mm}{2.5mm}\par\medskip}

\let\svthefootnote\thefootnote
\newcommand\freefootnote[1]{%
	\let\thefootnote\relax%
	\footnotetext{#1}%
	\let\thefootnote\svthefootnote%
}

\begin{document} 
	
	\title{\large{\bf The strong fractional choice number of triangle-free planar graphs}}
	
	\author{
		Xiaolan Hu\thanks{School of Mathematics and Statistics, and Hubei Key Laboratory of Mathematical Sciences, Central China Normal University, Wuhan 430079, China. Email: \texttt{xlhu@ccnu.edu.cn}.} 
		~and Rongxing Xu\thanks{School of Mathematical Sciences, Zhejiang Normal University, Jinhua, Zhejiang, 321000, China.  Email: \texttt{xurongxing@zjnu.edu.cn}.}
	}
	
	\maketitle
	
	\begin{abstract}  
		Let $a,b$ be positive integers with $a\ge b$.
		A graph $G$ is $(a,b)$-choosable if, for every assignment of lists 
		$L(v)$ of size $a$ to the vertices of $G$, there exists a choice of 
		subsets $C(v)\subseteq L(v)$ with $|C(v)|=b$ for each $v$ such that 
		$C(u)\cap C(v)=\emptyset$ whenever $uv\in E(G)$. We show that every triangle-free planar graph is $(15m,4m)$-choosable for any positive integer $m$. As an immediate consequence, the strong fractional choice number of triangle-free planar graphs is at most $15/4$. This appears to be the first non‑trivial upper bound on this parameter for this class of graphs. In particular, the case $m=1$ answers affirmatively a question posed by Jiang and Zhu in [J.~Combin.\ Theory Ser.~B, 2019].  
	\end{abstract}
	
	\section{Introduction}  
	All graphs considered in this paper are simple and loopless. Throughout the paper, we use $\mathbb{N}$ to denote the set of nonnegative integers, and $\mathbb{N}^G$ denote the set of mappings from $V(G)$ to $\mathbb{N}$. For a positive integer $k$, let $[k]=\{1,\ldots, k\}$.
	
	Suppose $G$ is a graph,  $f, g \in \mathbb{N}^G$ are two functions, with $f(v) \geq g(v)$ for each $v \in V(G)$. An \emph{$f$-list assignment} of $G$ is a mapping $L$ which assigns to each vertex $v$ of $G$ a set $L(v)$ of $f(v)$ integers as {\em permissible colors}. A \emph{$g$-fold coloring} of $G$ is a mapping $C$ which assigns to each vertex $v$ of $G$ a set $C(v)$ of $g(v)$ colors such that $C(v) \cap C(u)= \emptyset$ for any two adjacent vertices $u$ and $v$. If $g \equiv b$ (i.e., $g(v)=b$ for any $v\in V(G)$), then $(L, g)$-colorable is denoted by $(L, b)$-colorable. 
	If $L(v)=[a]$ for each $v \in V(G)$, then $(L, b)$-colorable is called  $(a,b)$-colorable.  The \emph{$b$-fold chromatic number} $\chi_b(G)$ of $G$ is the least $k$ such that $G$ is $(k,b)$-colorable. When $b=1$, $\chi_1(G)$ is the usual chromatic number, denoted by $\chi(G)$. The {\em fractional chromatic number}   of   $G$ is  defined as $	\chi_f(G)  = \inf \{\frac ab: \text{$G$ is $(a,b)$-colorable}\}$.

	Similarly, we say $G$ is {\em $(f,g)$-choosable} if for every $f$-list assignment $L$, $G$ is $(L, g)$-colorable. Some special cases of $(f,g)$-choosability have alternate names:
	\begin{itemize}
		\item If $g \equiv b$, then $(f,g)$-choosable is called $(f, b)$-choosable.
		\item If $f \equiv k$, then $(f,b)$-choosable is called $(k, b)$-choosable.
		\item If $b=1$, then $(f,1)$-choosable is called $f$-choosable.
		\item If $f \equiv k$, then $f$-choosable is called $k$-choosable.
	\end{itemize}
	The {\em choice number} $ch(G)$ of $G$ is the minimum $k$ such that
	$G$ is $k$-choosable. The {\em $b$-fold choice number}  $ch_b(G)$ of $G$
	is the minimum $k$ such that $G$ is $(k,b)$-choosable. The {\em fractional choice number}   of   $G$ is defined as $ch_f(G)   = \inf \{\frac ab: \text{$G$ is $(a,b)$-choosable}\}$. 
	The concept of list coloring of graphs was introduced in the 1970s by Vizing \cite{Vizing1976} and independently by Erd\H{o}s, Rubin and Taylor \cite{ERT1979}, and has been studied extensively in the literature. The area offers numerous challenging
	problems and has attracted increasing attention since the 1990s. 
	
	From the definitions we have,  for any graph $G$, $$\chi_f(G) \leq   \chi(G) \le ch(G) ~\text{and}~ \chi_f(G)  \le ch_f(G).$$ It is known that there exist bipartite graphs with arbitrary large choice number. However, Alon, Tuza and Voigt \cite{ATV1997} proved that $ch_f(G) = \chi_f(G)$ for every infinite graph $G$. Hence $ch_f(G)$ is not a genuinely new parameter. This motivates the concept of \emph{strong fractional choice number} of a graph, which was introduced in \cite{Zhu2017} and defined as 
	$$ch_f^s(G) = \inf\{r: G \text{ is $(a,b)$-choosable for all integers $a$ and $b$ with } a/b \geq r\}.$$ 
	The strong fractional choice number of a family $\mathcal{G}$ is the supremum of $ch_f^s(G)$ over all $G\in \mathcal{G}$. 
	It follows from the definition that $ch_f^s(G) \ge ch(G)-1$. This lower bound can be attained:  for example, Xu and Zhu \cite{XZ2023-2} showed that every bipartite $3$-choice critical graph $G$ satisfies  $ch_f^s(G)=2$, while its choice number is 3. For the upper bound, it remains open and asked by Xu and Zhu \cite{XZ2023} that whether $ch^s_f(G) \leq ch(G)$ for any graph $G$. 
	
	Let $a \geq b$ be two positive integers. It is straightforward that every $(a,b)$-colorable graph is also $(am,bm)$-colorable for every positive integer $m$. Erd\H{o}s, Rubin and Taylor \cite{ERT1979} asked  whether every graph that is $(a,b)$-choosable must also be $(am,bm)$-choosable. If the answer is positive, then it would imply that  $ch^s_f(G) \leq ch(G)$ for any graph $G$.  Tuza and Voigt \cite{TV1996-2m} proved this holds for $a=2$ and $b=1$. However,  for more general settings, Dvo\v{r}\'{a}k, Hu and Sereni \cite{DHS2019} proved that for any integer $k \geq 4$, there is a $k$-choosable graph which is not $(2k, 2)$-choosable. Nevertheless, it is still possible that $ch^s_f(G) \leq ch(G)$ for any graph $G$.
	
	It is well known that every planar graph is $4$-colorable by the Four-Color Theorem,  yet there exists planar graphs that are not $4$-choosable \cite{Voigt1993}. Zhu \cite{Zhu2017} strengthened this  by showing that for every positive integer $m$, there exists a planar graph which is not $(4m + \lfloor \frac{2m-1}{9}\rfloor, m)$-choosable. The value $4m + \lfloor \frac{2m-1}{9}\rfloor$ was later improved to $4m+  \lfloor \frac{m-1}{3}\rfloor$ by Xu and Zhu \cite{XZ2023}, and then to $4m + \lfloor \frac{2m-1}{5}\rfloor$ by Chappell \cite{Chappell2022}. On the other hand, Thomassen \cite{Thomassen1994} proved that every planar graph is $5$-choosable. Voigt \cite{Voigt1998} extended this by showing that every planar graph is $(5m,m)$-choosable for every positive integer $m$. Consequently, for the family of planar graphs $\mathcal{P}$, we have $$4+\frac{2}{5} \leq ch^s_f(\mathcal{P}) \leq 5.$$ It remains unknown if $ch^s_f(\mathcal{P}) < 5$,  this would follow if the following conjecture holds:
	\begin{conjecture}[\cite{Zhu2017}]
		\label{conj-planar}
		There exists a constant integer $m$ such that every planar graph is $(5m-1,m)$-choosable.
	\end{conjecture}
	
	By the Four-Color Theorem, every planar graph is $(4m,m)$-colorable for any positive integer $m$. Without using the Four-Color Theorem, Hilton, Rado and Scott \cite{HRS1973} proved that every planar graph is $(5m-1,m)$-colorable with $m=|V(G)|+1$. Cranston and Rabern \cite{CR2018} proved that every planar graph is $(9,2)$-colorable. However, the technique of contractions of nonadjacent vertices used in the proof \cite{CR2018} can not be applied to the study of list coloring, and Conjecture~\ref{conj-planar} remains largely open. 
	
	In this paper, we  focus on the strong fractional choice number of the family of  triangle-free planar graphs. The famous Gr\"{o}stzch Theorem \cite{Grotzsch} asserts that every triangle-free planar graph is $3$-colorable. Dvo\v{r}\'{a}k, Sereni and Volec \cite{DSV2015} proved that for every $n$-vertex triangle-free planar graph $G$, $\chi_f(G) \leq 3-\frac{1}{3n+1}$. Later, Dvo\v{r}\'{a}k and Hu \cite{DH2020} improved this bound to $3-\frac{1}{2n+1}$. However,  Voigt \cite{Voigt1995} showed that there exists a triangle-free planar graph that is not $3$-choosable. Jiang and Zhu \cite{JZ2019} further strengthened  this result by showing that for every positive integer $m$, there exists a triangle-free planar graph that is not $(3m+\lceil \frac{m}{17}\rceil, m)$-choosable. On the other hand, every triangle-free planar graph is $3$-degenerate, hence it is $(4m,m)$-choosable for any positive integer $m$. Consequently, for the family of triangle-free planar graphs $\mathcal{P}_3$, $$3+\frac{1}{17} \leq ch^s_f(\mathcal{P}_3) \leq 4.$$ It is unknown whether the upper bound is tight. Jiang and Zhu \cite{JZ2019} posed the following question, whose positive answer would imply $ch^s_f(\mathcal{P}_3) < 4$.
	\begin{question}[\cite{JZ2019}]
		\label{main-question}
		Does there exist a positive integer $k$ such that every triangle-free planar graph is $(4k-1,k)$-choosable ?
	\end{question}
	
	Note that the smallest possible value $k$ in the question above is $k=2$. The special case $m=1$ of our main result below answers this question positively with $k=4$. 
	
	\begin{theorem}
		\label{main-thm}
		Every triangle-free planar graph is $(15m,4m)$-choosable for any positive integer $m$. As a corollary, $3+\frac{1}{17} \leq  ch^s_f(\mathcal{P}_3) \leq 3 + \frac{3}{4}$.
	\end{theorem}

	The paper is organized as follows. In Section~\ref{sec-pre}, we introduce a new language to describe how to extend a partial coloring of a graph to a complete coloring, and provide several useful lemmas. In Section~\ref{sec-RC}, we show that certain  subgraphs are reducible in a minimum counterexample. In Section~\ref{sec-discharging}, we apply the discharging method to obtain a contradiction, thereby proving Theorem~\ref{main-thm}.
	
	\section{Preliminaries}
	\label{sec-pre}
	For a graph $G$ and a vertex $u \in V(G)$, $N_G(u)$ is the set of neighbors of $u$, and $d_G(u)=|N_G(u)|$ is the degree of $u$. For brevity, we simply write $d(u)$ when the underlying graph is clear from the context. For a subset $S \subseteq V(G)$, $G-S$ is the graph obtained from $G$ by deleting the vertices in $S$ together with all edges incident to them, and $G[S]$ is the subgraph of $G$ induced by $S$. For a subgraph $H$ of $G$ and a vertex $u \in V(G)$ (not necessarily in $V(H)$), $N_H(u)=N_G(u)\cap V(H)$ and $d_H(u)=|N_H(u)|$.
	
	For $f \in \mathbb{N}^G$ and a subset $U \subseteq V(G)$,
	let $f|_U$ denote the restriction of $f$ to $U$.
	When no confusion arises, we simply write $f$ instead of $f|_U$.
	The same convention applies to the restriction $L|_U$ of a list assignment $L$.
	Let $\mathcal{T}$ denote the collection of triples $(G,L,g)$, where $G$ is a graph,
	$L$ is a list assignment of $G$, and $g \in \mathbb{N}^G$ specifies the number of
	colors required at each vertex. 
	
	Unlike ordinary $1$-fold coloring, where a vertex can be discarded once it is
	assigned a color, in multiple list coloring a vertex $v$ with $g(v)>0$ must
	remain under consideration until all its color demands are satisfied.
	Accordingly, colors may be assigned incrementally, while both the list
	assignment and the demand function are updated to reflect these partial
	assignments. To formalize this process, we describe such local updates as operations acting on
	triples $(G,L,g) \in \mathcal{T}$, which allows us to track changes in a
	systematic way and avoids repeating routine update arguments in later proofs.
	
	\subsection{Local reduction operations}
	
	We begin by formalizing the operation of assigning colors to a subset of vertices,
	while leaving the remaining demands to be satisfied later.
	
	\begin{definition}\label{def-parcol}
		Let $(G,L,g)\in\mathcal T$ and let $\phi$ be a partial $(L,g)$-coloring of $G$
		on $U\subseteq V(G)$.  (We interpret $\phi(v)=\emptyset$ for $v\notin U$.)
		The {\rm partial coloring operation}
		\[
		\mathsf{ParCol}_{\phi,U}:(G,L,g)\mapsto (G,L',g')
		\]
		is defined by, for each $v\in V(G)$,
		\[
		g'(v)=g(v)-|\phi(v)|,\quad \text{and} \quad
		L'(v)=L(v)\setminus\Bigl(\phi(v)\cup\bigcup_{u\in N_G(v)}\phi(u)\Bigr).
		\]
		We say $\mathsf{ParCol}_{\phi,U}$ is {\bf legal} if $|L'(v)|\ge g'(v)$ for all $v\in V(G)$.
	\end{definition} 
	
	\begin{remark}
		\label{rmk1}
		We mainly use the following special cases of the partial coloring operation
		$\mathsf{ParCol}_{\phi,U}$.
		\begin{enumerate}[label=(\arabic*)]
			\item \label{rmk1-1}
			If $U=\{u\}$ and $\phi(u)=A$, we write it as 
			$\mathsf{ParCol}(u,A)$.
			
			\item \label{rmk1-2}
			If $U=\{u\}$, $uv \in E(G)$, and $\phi$ satisfies
			$\phi(u)\subseteq L(u)\setminus L(v)$ and $|\phi(u)|=k$,
			we write it as $\mathsf{ParCol}(u\mid v,k)$.
			
			\item \label{rmk1-3}
			If $U=\{u_1,u_2\}$, $v$ is a common neighbor of $u_1$ and $u_2$, and $\phi$ satisfies
			$\phi(u_1)=S\cup R$ and $\phi(u_2)=T\cup R$ for some
			$S\subseteq L(u_1)\setminus L(v)$,
			$T\subseteq L(u_2)\setminus L(v)$,
			$R\subseteq L(u_1)\cap L(u_2)\cap L(v)$,
			and $|S|+|T|+|R|=k$,
			we write it as $\mathsf{ParCol}(\{u_1,u_2\}\mid v, k^*)$.
		\end{enumerate}
	\end{remark}

	Note that in Remark~\ref{rmk1}\ref{rmk1-2}, the partial coloring guarantees that the vertex $u$
	is assigned $k$ colors without reducing the available colors of its neighbor $v$,
	and hence saves $k$ colors for $v$.
	In Remark~\ref{rmk1}\ref{rmk1-3}, the partial coloring assigns a total of $k+|R|$ colors to $u_1$ and $u_2$, while using only $|R|$ colors from $L(v)$, and therefore saves $k$ colors
	with respect to $v$.
	
	In many situations, a vertex can be safely removed once its list is large
	enough compared with the remaining  color demands of its neighbors.
	This allows us to reduce the graph while preserving the possibility of
	extending a coloring.
	
\begin{definition}
		The {\rm degenerate deletion operation}
		$\mathsf{DegDel}(u):\mathcal T\to\mathcal T$
		is defined by
		\[
		\mathsf{DegDel}(u)(G,L,g)=(G',L',g'),
		\]
		where $G'=G-u$, $L'=L|_{V(G')}$, and $g'=g|_{V(G')}$.
		We say that $\mathsf{DegDel}(u)$ is {\bf legal} if
		$|L(u)| \ge g(u) + \sum_{v\in N_G(u)} g(v)$.
	\end{definition}

	\begin{definition}
		A {\bfseries reduction scheme} is an operator
		$\Omega=\mathsf{Red}(\theta_1,\theta_2, \ldots, \theta_k):\mathcal T\to\mathcal T$,
		where for each $i$, one of the following holds:
		$\theta_i=\langle u\rangle$, representing the degenerate deletion operation
		$\mathsf{DegDel}(u)$;
		$\theta_i=\langle U, \phi\rangle$, representing the partial coloring operation
		$\mathsf{ParCol}_{\phi, U}$.  
		The reduction scheme $\Omega$ is defined recursively by
		\[
		\mathsf{Red}(\langle u\rangle)(G,L,g)=\mathsf{DegDel}(u)(G,L,g),
		\quad
		\mathsf{Red}(\langle U, \phi\rangle)(G,L,g)=\mathsf{ParCol}_{\phi, U}(G,L,g),
		\] 
		and for $k\ge 2$,
		\[
		\mathsf{Red}(\theta_1,\theta_2, \ldots, \theta_k)(G,L,g)
		=
		\mathsf{Red}(\theta_k)\bigl(
		\mathsf{Red}(\theta_1,\theta_2, \ldots, \theta_{k-1})(G,L,g)
		\bigr).
		\]
		We say that the reduction scheme
		$\mathsf{Red}(\theta_1, \ldots, \theta_k)$ is {\bf legal}
		for $(G,L,g)\in\mathcal T$ if
		$\mathsf{Red}(\theta_1, \ldots, \theta_{k-1})$ is legal for $(G,L,g)$ and
		$\theta_k$ is legal for
		$\mathsf{Red}(\theta_1, \ldots, \theta_{k-1})(G,L,g)$.
	\end{definition}
	
	\begin{remark}
		In the construction of reduction schemes, we will use the following
		special forms of $\langle U,\phi\rangle$ corresponding to the partial coloring
		operations introduced in Remark~\ref{rmk1}:
		\[
		\theta = \langle u, A\rangle, \quad
		\theta = \langle u \mid v, k\rangle \quad\text{and}\quad
		\theta = \langle \{u_1,u_2\} \mid v, m^*\rangle,
		\]
		which represent $\mathsf{ParCol}(u,A)$, $\mathsf{ParCol}(u\mid v,k)$, and
		$\mathsf{ParCol}(\{u_1,u_2\}\mid v, m^*)$, respectively.
	\end{remark}
	
	\begin{observation}\label{obs-reduce}
		Let $(G,L,g)\in\mathcal T$ and let
		$\Omega=\mathsf{Red}(\theta_1,\theta_2, \ldots, \theta_k)$
		be a reduction scheme that is legal for $(G,L,g)$.
		If $\Omega$ removes all vertices of $G$, then $G$ admits an $(L,g)$-coloring.
	\end{observation}
	
	\subsection{Nice subgraphs and associated functions}
	
	In this section, we develop a framework for proving reducibility of configurations.
	We introduce a notion of \emph{nice} subgraphs and associate with each such subgraph
	a pair of functions $(f_H,g_H)$. We then prove a key lemma that will be used repeatedly
	in reducibility arguments.

	\begin{definition}
		For a subgraph $H$ of $G$, let
		\[
		U_H=\{u\in V(H)\mid d_H(u)\ge d_G(u)-1\},
		\]
		and let $F_H$ be the graph induced by $V(H)\setminus U_H$.
		We say that $H$ is \emph{nice} if $U_H\neq\emptyset$, $d_H(u)\ge d_G(u)-2$ for each $u \in V(H)$, and  each component of $F_H$ is isomorphic to a complete bipartite graph $K_{1,3}$ or an $\ell$-vertex path $P_\ell$ for some $\ell\in [4]$. 
	\end{definition}
	
	Note that if $H$ is nice, then $V(F_H)=\{u\in V(H)\mid d_G(u)-d_H(u)=2\}$. 
	
	\begin{definition}
		\label{def-(f_H,g_H)}
		For a nice subgraph $H$ of $G$, let $f_H,g_H \in \mathbb{N}^H$ be two functions
		defined as follows.
		\begin{enumerate}[label=(\arabic*)]
			\item \label{P1}
			If $u$ lies in a component of $F_H$ isomorphic to $P_1$, then
			$f_H(u)=7m$ and $g_H(u)=4m$.
			\item \label{P2}
			If $u$ lies in a component of $F_H$ isomorphic to $P_2$, then
			$f_H(u)=6m$ and $g_H(u)=3m$.
			\item \label{P3}
			If $u$ lies in a component of $F_H$ isomorphic to $P_3$, then
			$f_H(u)=5m$, and
			$g_H(u)=2m$ if $d_{F_H}(u)=2$, while
			$g_H(u)=3m$ if $d_{F_H}(u)=1$.
			\item \label{P4}
			If $u$ lies in a component of $F_H$ isomorphic to $P_4$, then
			$(f_H,g_H)(u)=(5m,3m)$ if $d_{F_H}(u)=1$, and $(f_H,g_H)(u)=(4m,2m)$ if $d_{F_H}(u)=2$.
			\item \label{K13}
			If $u$ lies in a component of $F_H$ isomorphic to $K_{1,3}$, then
			$f_H(u)=4m$, and
			$g_H(u)=m$ if $d_{F_H}(u)=3$, while
			$g_H(u)=3m$ if $d_{F_H}(u)=1$.
			\item \label{others}
			If $u\in V(H)\setminus V(F_H)$, then $g_H(u)=4m$, and
			\[
			f_H(u)=15m-4m(d_G(u)-d_{H}(u))
			-\sum_{v\in N_{F_H}(u)}(4m-g_H(v)).
			\]
		\end{enumerate}
	\end{definition}

	\begin{lemma}
		\label{key lemma}
		Let $G$ be a minimal triangle-free planar graph that is not $(15m,4m)$-choosable for some integer $m$. If $H$ is a nice subgraph of $G$, then $H$ is not $(f_H,g_H)$-choosable. Consequently, by Observation~\ref{obs-reduce}, 
		there is no legal reduction scheme that deletes all vertices of $H$. 
	\end{lemma}
	
	\begin{proof} 
		If $f_H(v) < g_H(v)$ for some $v \in V(H)$, then $H$ is not $(f_H,g_H)$-choosable obviously. Thus we assume that $f_H(v) \geq g_H(v)$ for every vertex $v \in V(H)$.
		
		Let $L$ be a $15m$-list assignment of $G$. By the minimality of $G$, the graph $G-U_H$ has an $(L,4m)$-coloring $\phi$ since $U_H\neq\emptyset$. Define a list  assignment $L_0$ of $H$ by $$L_0(u) = L(u) \setminus \{\phi(v) \mid v \in N_{G - V(H)}(u)\}, ~~u\in V(H).$$ Thus by deleting extra colors not in $\phi(u)$, we assume that
		\begin{itemize}
			\item $|L_0(u)| =  7m$ for each $u \in V(F_H)$, and
			\item $|L_0(v)| = 15m-4m|N_{G-U_H}(v)| = 15m-4m(d_G(v)-d_H(v))$ for each $v \in V(H)\setminus V(F_H)$.
		\end{itemize}    
		Moreover, since $F_H$ is a subgraph of $G-U_H$ and $\phi$ is an $(L,4m)$-coloring of $G-U_H$, it follows that $F_H$ has an $(L_0,4m)$-coloring. If $H$ had an $(L_0,4m)$-coloring $\phi'$, then $\phi\cup\phi'$ would be an $(L,4m)$-coloring of $G$, a contradiction. Hence $H$ is not $(L_0,4m)$-colorable. 
		
		If $F_H$ is an independent set, then by Definition \ref{def-(f_H,g_H)}\ref{P1} and~\ref{others}, $L_0$ is an $f_H$-list assignment and $g_H(u)=4m$ for each $u \in V(H)$. Consequently,  $H$ is not $(f_H,g_H)$-colorable, as desired.

		Now assume that some component of $F_H$ is isomorphic to $K_{1,3}$ or to $P_\ell$ for some $\ell\in\{2,3,4\}$, and let $F_H^*$ be the union of all such components. In the remainder of the proof, we construct a partial coloring $\psi$ of $F_H^*$ such that the operation $\mathsf{ParCol}_{\psi,V(F_H^*)}$ is legal and $|L_1(v)| \ge f_H(v)$ for each $v \in V(H)$, with $g_1=g_H$, where
		\[(H,L_1,g_1)=\mathsf{ParCol}_{\psi,V(F_H^*)}(H,L_0,4m).\]
		This implies that $H$ is not $(L_1, g_1)$-colorable, and hence is not $(f_H,g_H)$-choosable.
		Therefore, it suffices to define $\psi$ on each component of $F_H^*$.
		
		Let $F$ be such a component. We define $\psi$ on $F$ by considering four cases.

		\medskip
		\noindent 
		{\bf Case 1.} $F\cong P_2$, say $F=v_1v_2$.
		\medskip
		
		Since $F$ admits an $(L_0,4m)$-coloring, we have $|L_0(v_1)\cup L_0(v_2)|\ge 8m$. As $|L_0(v_1)|=|L_0(v_2)|=7m$, there exists an $m$-set $A\subseteq L_0(v_1)\setminus L_0(v_2)$ and an $m$-set $B\subseteq L_0(v_2)\setminus L_0(v_1)$. Let $\psi(v_1)=A$ and $\psi(v_2)=B$. It is easy to see that $|L_1(v_i)|\ge 6m=f_H(v_i)$ and $g_1(v_i)=3m$ for each $i\in\{1,2\}$, satisfying Definition~\ref{def-(f_H,g_H)}\ref{P2} and~\ref{others}.

		\medskip
		\noindent 
		{\bf Case 2.} $F\cong P_3$, say $F=v_1v_2v_3$.
		\medskip
		
		Since $F$ admits an $(L_0,4m)$-coloring, there exist an $m$-set
		$A\subseteq L_0(v_1)\setminus L_0(v_2)$ and an $m$-set
		$C\subseteq L_0(v_3)\setminus L_0(v_2)$.
		Moreover, there exist an $m$-set $B_1\subseteq L_0(v_2)\setminus L_0(v_1)$
		and an $m$-set $B_2\subseteq L_0(v_2)\setminus L_0(v_3)$.
		Choose a $2m$-set $B\subseteq L_0(v_2)$ such that $B_1\cup B_2\subseteq B$.
		Let $\psi(v_1)=A$, $\psi(v_2)=B$, and $\psi(v_3)=C$.
		It is clear that $|L_1(v_i)|\ge 5m=f_H(v_i)$ for each $i \in [3]$.
		Moreover, $g_1(v_1)=g_1(v_3)=3m$ and $g_1(v_2)=2m$, satisfying
		Definition~\ref{def-(f_H,g_H)}\ref{P3} and~\ref{others}.

		\medskip
		\noindent 
		{\bf Case 3.} $F\cong P_4$, say $F=v_1v_2v_3v_4$.
		\medskip
		
		Since $F$ admits an $(L_0,4m)$-coloring, we can choose, for each $i\in [3]$, an $m$-set $A_i\subseteq L_0(v_i)\setminus L_0(v_{i+1})$, and for each $j\in\{2,3,4\}$, an $m$-set
		$B_j\subseteq L_0(v_j)\setminus L_0(v_{j-1})$. Moreover, these sets can be chosen so that $A_3\cap B_2=\emptyset$. Indeed, fix an $(L_0,4m)$-coloring $\varphi$ of $F$. Since $|L_0(v_i)|=7m$, the disjointness of $\varphi(v_1)$ and $\varphi(v_2)$ implies $|\varphi(v_2)\setminus L_0(v_1)|\ge m$, and similarly $|\varphi(v_3)\setminus L_0(v_4)|\ge m$. Hence we may choose $B_2\subseteq \varphi(v_2)\setminus L_0(v_1)$ and $A_3\subseteq \varphi(v_3)\setminus L_0(v_4)$, which are disjoint because $\varphi(v_2)\cap \varphi(v_3)=\emptyset$. Note that the sets assigned to consecutive vertices are disjoint, and $A_2\cap B_3=\emptyset$. Now pick a $2m$-set $B\subseteq L_0(v_2)\setminus A_3$ containing $A_2\cup B_2$, and a $2m$-set
		$C\subseteq L_0(v_3)\setminus B$ containing $A_3\cup B_3$. Set $\psi(v_1)=A_1$, $\psi(v_2)=B$, $\psi(v_3)=C$ and $\psi(v_4)=B_4$. It is not hard to verify that $|L_1(v_i)| \geq 5m = f_H(v_i)$ and $g_1(v_i) = 3m$ for each $i \in \{1, 4\}$,  and $|L_1(v_j)| \geq 4m = f_H(v_j)$ and $g_1(v_j) = 2m$ for each $j \in \{2, 3\}$, satisfying  Definition~\ref{def-(f_H,g_H)}\ref{P4} and~\ref{others}.

		\medskip
		\noindent 
		{\bf Case 4.} $F\cong K_{1,3}$ with center $v_4$ and leaves $v_1, v_2, v_3$
		\medskip
		
		Since $F$ admits an $(L_0,4m)$-coloring, for each $i\in [3]$ there exist an $m$-set $A_i\subseteq L_0(v_i)\setminus L_0(v_4)$ and an $m$-set $B_i\subseteq L_0(v_4)\setminus L_0(v_i)$. Choose a $3m$-set $B\subseteq L_0(v_4)$ such that $B_1\cup B_2\cup B_3\subseteq B$. Let $\psi(v_i)=A_i$ for each $i\in[3]$ and $\psi(v_4)=B$. It is straightforward to verify that $|L_1(v_i)|\ge 4m=f_H(v_i)$ for each $i \in [4]$; moreover, $g_1(v_1)=g_1(v_2)=g_1(v_3)=3m$ and $g_1(v_4)=m$, satisfying Definition~\ref{def-(f_H,g_H)}\ref{K13} and~\ref{others}.
		
		Finally, for every vertex $u \in U_H$, we have $g_1(u) = 4m =g_H(u)$, and
		\[
		|L_1(u)| \geq |L_0(u)|-|\bigcup_{v \in N_{F^*_H}(u)}\psi(v)| \geq
		15m-4m(d_G(u)-d_{H}(u)) 
		-\sum_{v\in N_{F_H}(u)}(4m-g_H(v))= f_H(u).
		\]  
		
		Thus $|L_1(v)| \geq f_H(v)$ and $g_1(v) = g_H(v)$ for any $v \in V(H)$. For each $v\in V(H)$, choose an $f_H(v)$-subset $L_2(v)$ of $L_1(v)$. Since $H$ is not $(L_1,g_H)$-colorable, it is not $(L_2,g_H)$-colorable. Hence $H$ is not $(f_H,g_H)$-choosable, as desired.
	\end{proof}
	
	\subsection{Two key techniques for constructing reduction schemes}
	
	\begin{lemma}
		\label{lem-three sets}
		Suppose $m$ is a positive integer, $A,B,C$ are three sets. If $|A|+|B| \geq |C|+ m$, then there exist sets $S \subseteq A\setminus C$, $T \subseteq B\setminus C$ and $R \subseteq A \cap B \cap C$ such that $|S|+|T|+|R|=m$.
	\end{lemma}  
	\begin{proof}
		By the assumption $|A|+|B|\ge |C|+m$ and the inclusion-exclusion principle,
		we have
		\begin{align*} 
			m	&\le |A|+|B|-|C|  = |A\cup B| + |A\cap B| - |C| \\
			&= |(A\cup B)\cap C| + |(A\cup B)\setminus C|
			+ |A\cap B| - |C| \\
			&\le |(A\cup B)\setminus C| + |A\cap B| \\
			&= |(A\setminus C)\cup (B\setminus C)| + |A\cap B| \\
			&= |A\setminus C| + |B\setminus C|
			- |(A\cap B)\setminus C| + |A\cap B| \\
			&= |A\setminus C| + |B\setminus C| + |A\cap B \cap C|.
		\end{align*}
		The above inequality guarantees the existence of subsets
		$S\subseteq A\setminus C$, $T\subseteq B\setminus C$ and
		$R\subseteq A\cap B \cap C$ satisfying $|S|+|T|+|R|=m$. 
	\end{proof}
	
	Suppose that we aim to prove that a graph $H$ is $(L_H,g_H)$-colorable, and that
	there exists a vertex $v\in V(H)$ satisfying
	\begin{equation}
		\label{eq-almost-deg}
		|L_H(v)| + m \ge g_H(v) + \sum_{u\in N_H(v)} g_H(u). 
	\end{equation} 
	for some integer $m \ge 1$. Note that $\langle v\rangle$ maybe not legal for $(H,L_H,g_H)$.
	
	If there exists a neighbor $u$ of $v$ such that
	\begin{equation}
		\label{eq-u mid v}
		|L_H(u)\setminus L_H(v)| \ge m, 
	\end{equation} 
	then we may first apply  $\langle u \mid v, m\rangle$, and then apply  $\langle v\rangle$. It is easy to see that if $\langle u \mid v, m\rangle$ is legal for $(H,L_H,g_H)$, then $\langle v\rangle$  is  legal for $\mathsf{Red}(\langle u \mid v, m\rangle)(H,L_H,g_H)$.
	
	Similarly, if there exist two neighbors $u_1,u_2$ of $v$ such that
	\begin{equation}
		\label{eq-two sum}
		|L_H(u_1)|+|L_H(u_2)| \ge |L_H(v)| + m,
	\end{equation}  
	then by Lemma~\ref{lem-three sets}, we may first apply  
	$\langle \{u_1,u_2\}\mid v, m^*\rangle$, and then apply  
	$\langle v\rangle$.  
	We now show that if $\langle \{u_1,u_2\}\mid v, m^*\rangle$ is legal for $(H,L_H,g_H)$, then $\langle v\rangle$  is  legal  for $\mathsf{Red}(\langle \{u_1,u_2\}\mid v, m^*\rangle)(H,L_H,g_H)$. 
	Indeed, assume that $(H',L', g') = \mathsf{ParCol}(\{u_1,u_2\}\mid v, m^*)(H,L_H, g_H)$, and the corresponding coloring is $\phi$, with $\phi(u_1)= S\cup R$ and $\phi(u_2) = T \cup R$, where $S\subseteq L(u_1)\setminus L(v)$, $T \subseteq L(u_2)\setminus L(v)$, and $R \subseteq L(u_1)\cap L(u_2) \cap L(v)$, and $|S|+|R|+|T|=m$. Since $S \cap R =\emptyset$ and $T \cap R =\emptyset$, we have $g'(u_1)=g_H(u_1)-|S|-|R|$, $g'(u_2)=g_H(u_2)-|T|-|R|$  and that $\langle v\rangle$  is legal by the following:
	\begin{align*}
		|L'(v)| & = |L_H(v)|-|R| \geq  g_H(v) + g_H(u_1) + g_H(u_2)-m -|R| + \sum_{w\in N_H(v)\setminus \{u_1,u_2\}} g_H(w)  \\
		& =   g_H(v) + g_H(u_1)+g_H(u_2)- |S|- |T| - 2|R| + \sum_{w\in N_H(v)\setminus \{u_1,u_2\}} g_H(w) \\
		& = \big(g_H(u_1) - |S|-|R|\big) + \big(g_H(u_2) - |T|-|R|\big) + g_H(v) +  \sum_{w \in N_H(v)\setminus \{u_1,u_2\}} g_H(w)   \\
		& = g'(v) + g'(u_1) + g'(u_2) + \sum_{w\in N_H(v)\setminus \{u_1,u_2\}} g'(w).
	\end{align*} 
	
	The arguments above illustrate two standard ways of applying the saving operations
	$\langle u\mid v, m\rangle$ and $\langle \{u_1,u_2\}\mid v, m^*\rangle$,
	which are often used in combination with $\langle v\rangle$.
	In the remainder of the paper, whenever one of these reduction patterns is used,
	we will not repeat the verification, even though $\langle v\rangle$ does not always immediately follow.
	Instead, it will be clear from the context that the corresponding inequalities
	\eqref{eq-almost-deg}, \eqref{eq-u mid v}, and \eqref{eq-two sum} are satisfied.

	\section{Reducible configurations}
	\label{sec-RC}
	
	For a positive integer $k$, a vertex $v$ is called a \emph{$k$-vertex} if $d(v)=k$,
	a \emph{$k^+$-vertex} if $d(v)\ge k$, and a \emph{$k^-$-vertex} if $d(v)\le k$.
	For a vertex $v$, a neighbor $u$ of $v$ is called a \emph{$k$-neighbor} if $u$ is a
	$k$-vertex. Similarly, we define $k^+$-neighbors and $k^-$-neighbors.
	Moreover, for integers $k$ and $t$, a vertex $v$ is called a \emph{$k_t$-vertex}
	if $v$ is a $k$-vertex and has exactly $t$ neighbors of degree $3$, and a
	\emph{$k_{\ge t}$-vertex} if $v$ is a $k$-vertex and has at least $t$ neighbors of
	degree $3$. 
	
	Let $d_1,d_2, \ldots, d_k$ be $k$ integers (not necessarily distinct).
	A \emph{$(d_1,d_2, \ldots, d_k)$-path} is a path $v_1v_2\cdots v_k$ such that
	$d(v_i)=d_i$ for all $i\in [k]$.
	A \emph{$(d_1,d_2, \ldots, d_k)$-cycle} is a cycle $v_1v_2\cdots v_kv_1$ such that
	$d(v_i)=d_i$ for all $i\in [k]$.
	Here, each $d_i$ may be replaced by any of the degree notations defined above to
	specify a degree condition.
	For example, a $(4_1,5_{\ge 2},3)$-path is a path $v_1v_2v_3$ such that $v_1$ is a
	$4_1$-vertex, $v_2$ is a $5_{\ge 2}$-vertex, and $v_3$ is a $3$-vertex.
	
	Let $G$ be a minimal counterexample to Theorem~\ref{main-thm} with respect to the number of vertices,
	and let $L$ be a fixed $15m$-list assignment of $G$.
	Then $G$ is not $(L,4m)$-colorable, whereas every proper subgraph
	$G'\subsetneq G$ is $(L|_{V(G')},4m)$-colorable.
	
	We now proceed to describe a family of reducible configurations.  
	In most of the proofs below, we first define a subgraph $H$ of $G$.
	It is straightforward to verify that $H$ is nice, and we omit this routine verification.
	By Definition~\ref{def-(f_H,g_H)}, we then specify the functions $(f_H,g_H)$ on $H$.
	Throughout these proofs, $L_H$ denotes an arbitrary $f_H$-list assignment of $H$.

	\begin{claim}
		\label{FC-min-degree}
		$G$ has minimum degree at least $3$. 
	\end{claim} 
	\begin{proof}
		Suppose for a contradiction that there exists a vertex $v\in V(G)$ such that $d(v)\leq 2$.
		By the minimality of $G$, the graph $G-v$ has an $(L, 4m)$-coloring $\phi$. In $\phi$,
		each neighbor of $v$ has chosen $4m$ distinct colors, there remain at least  $15m-2\times 4m=7m$ colors in $L(v)$ that are not used by any neighbors of $v$ under $\phi$. Hence we can choose a set of $4m$ colors from those available colors and extend $\phi$ to an $(L, 4m)$-coloring of $G$, a contradiction.
	\end{proof}
	
	\begin{claim}
		\label{FC-star}
		For each $k\in\{3,4,5,6\}$, every $k$-vertex of $G$ has at most $k-2$ neighbors of degree $3$. 
		In particular, $G$ contains no $(3,3,3)$-path. 
	\end{claim} 
	\begin{proof}
		Suppose to the contrary,  $u$ is a $k$-vertex of $G$ and $v_1, v_2, \ldots, v_{k-1}$ are $3$-neighbors of $u$. Let $H=G[u, v_1,v_2, \ldots, v_{k-1}]$. By Definition~\ref{def-(f_H,g_H)},  $(f_H,g_H)(u)=(11m,4m)$, and
		$(f_H,g_H)(v_i)=(7m,4m)$ for each $i\in\{1,2, \ldots, k-1\}$.  
		If $k \in \{3,4,5\}$, then let \[\Omega=\mathsf{Red}\bigl(
		\langle u\mid v_1, m\rangle,\ \langle v_1\rangle,\ \langle u\mid v_2,m\rangle,\
		\langle v_2\rangle,\ \ldots,\  \langle u\mid v_{k-2},m\rangle,\ \langle u\rangle,\ \langle v_{k-1} \rangle 
		\bigr).\]
		If $k=6$, let 
		\[
		\Omega=\mathsf{Red}\bigl(
		\langle \{v_1, v_2\} \mid u, m^*\rangle,\   
		\langle u\mid v_3,m\rangle,\  \langle v_3\rangle,\ \langle u\mid v_4,m\rangle,\
		\langle v_4\rangle,\  \langle u\mid v_5, m\rangle,\ \langle v_5\rangle,\ \langle u\rangle,\ \langle v_1\rangle,\ \langle v_2\rangle
		\bigr).\] In each case, it is easy to verify that $\Omega$ is legal for $(H,L_H,g_H)$
		and deletes all vertices of $H$, which contradicts Lemma~\ref{key lemma}. 
	\end{proof} 
	
	\begin{claim}
		\label{FC-k_{k-2}-3_1}
		Suppose that $k\in\{4,5,6\}$. 
		If a $k$-vertex of $G$ has $k-2$ neighbors of degree $3$, 
		then none of these $k-2$ vertices has a $3$-neighbor. 
		In particular, $G$ contains no $(3,3,4,3)$-path. 
	\end{claim}
	\begin{proof}
		Assume to the contrary that there exists a $k$-vertex $u$ having $k-2$ neighbors
		$v_1,v_2,\ldots,v_{k-2}$ of degree $3$, and that one of these vertices, say $v_1$, has a
		$3$-neighbor $u_1$. By Claim~\ref{FC-star}, $u_1v_i \notin E(G)$ for each $2 \leq i \leq k-2$. Let $H=G[u,u_1,v_1, v_2, \ldots, v_{k-2}]$.    
		By Definition~\ref{def-(f_H,g_H)}, according to the values of $k$, we have that
		$(f_H,g_H)(u_1)=(7m,4m)$, $(f_H,g_H)(v_1)=\big((14-k)m,4m\big)$,
		$(f_H,g_H)(u)=((10-k)m,(7-k)m)$, and
		$(f_H,g_H)(v_i)=\big((10-k)m,3m\big)$ for each $i\in\{2,\ldots,k-2\}$.  
		Then 
		$$\mathsf{Red}\bigl(
		\langle v_2\rangle,\ \langle v_3\rangle,\  \ldots,\ \langle v_{k-2}\rangle,\  
		\langle v_1 \mid u_1,m\rangle,\ \langle u_1\rangle,\ \langle v_1\rangle,\ \langle u\rangle 
		\bigr)$$
		is legal for $(H,L_H, g_H)$ and deletes all vertices of $H$, which contradicts Lemma~\ref{key lemma}. 
	\end{proof}

	\begin{claim}
		\label{FC-(4-,4-,4-,3-)}
		$G$ contains no $(4^-,4^-,4^-,3)$-cycle. 
	\end{claim}
	\begin{proof}
		Assume to the contrary that $H=v_1v_2v_3v_4v_1$ is a cycle of $G$ with 
		$d(v_1)=3$ and $d(v_i)\in\{3,4\}$ for each $i\in\{2,3,4\}$. 
		By Claim~\ref{FC-star}, we cannot have $d(v_2)=d(v_4)=3$. 
		Without loss of generality, assume $d(v_2)=4$. 
		Again, we cannot have $d(v_3)=d(v_4)=3$. 
		We now consider the following three cases to determine $(f_H,g_H)$ by Definition~\ref{def-(f_H,g_H)}. 
		\begin{itemize}
			\item If $d(v_3)=4$ and $d(v_4)=4$, then 
			$(f_H,g_H)(v_1)=(9m,4m)$,
			$(f_H,g_H)(v_i)=(5m,3m)$ for $i\in\{2,4\}$, and
			$(f_H,g_H)(v_3)=(5m,2m)$. 
			\item If $d(v_3)=4$ and $d(v_4)=3$, then
			$(f_H,g_H)(v_i)=(10m,4m)$ for $i\in\{1,4\}$ and
			$(f_H,g_H)(v_j)=(6m,3m)$ for $j\in\{2,3\}$. 
			\item If $d(v_3)=3$ and $d(v_4)=4$, then
			$(f_H,g_H)(v_i)=(11m,4m)$ for $i\in\{1,3\}$ and
			$(f_H,g_H)(v_j)=(7m,4m)$ for $j\in\{2,4\}$.
		\end{itemize} 
		
		For the first two cases, let $\Omega= \mathsf{Red}\bigl(\langle \{v_2, v_4\} \mid v_1, m^*\rangle,\
		\langle v_1 \rangle,\ \langle v_2\rangle,\ \langle v_4\rangle,\ \langle v_3\rangle\bigr)$. 
		For the third case, let $\Omega= \mathsf{Red}\bigl(\langle \{v_2, v_4\} \mid v_1, m^*\rangle,\
		\langle v_1 \rangle,\ \langle v_3 \mid v_2, m\rangle,\ \langle v_2\rangle,\ \langle v_3\rangle,\ \langle v_4\rangle \bigr)$. In each case, $\Omega$ is legal for $(H,L_H, g_H)$ and deletes all vertices of $H$, which contradicts Lemma~\ref{key lemma}.  
	\end{proof} 
	
	\begin{claim}
		\label{FC-5_3/4_24_1}
		No $4_{\ge 1}$-vertex is adjacent to a $4_2$- or $5_3$-vertex. 
	\end{claim}
	\begin{proof}
		Assume to the contrary, $v_1$ is a $4_{\geq 1}$-vertex and $v_2$ is a $k_{k-2}$-vertex, where $k \in \{4,5\}$. Let $v_0$ be a $3$-neighbor of $v_1$, and let $v_3, \ldots, v_k$ be the $3$-neighbors of $v_2$. By Claim~\ref{FC-k_{k-2}-3_1}, $v_0v_i \notin E(G)$ for each $i \in \{3, \ldots, k\}$. Let $H=G[v_0, v_1,\ldots, v_k]$. By Definition~\ref{def-(f_H,g_H)}, $(f_H,g_H)(v_i)=(6m,3m)$ for $i\in\{0, 1\}$, $(f_H,g_H)(v_2)=(10m,4m)$, and $(f_H,g_H)(v_j)=(7m,4m)$ for $3 \leq j \leq k$.  
		Then
		$$\mathsf{Red}\bigl(
		\langle v_0\rangle,\   
		\langle v_2 \mid v_3,m\rangle,\ \langle v_3\rangle,\ \langle v_2 \mid v_4,m\rangle,\ \langle v_4\rangle,\ \ldots,\  \langle v_2 \mid v_k, m\rangle,\ \langle v_k\rangle,\  \langle v_2\rangle,\ \langle v_1\rangle
		\bigr)
		$$
		is legal for $(H,L_H, g_H)$ and deletes all vertices of $H$, which contradicts Lemma~\ref{key lemma}. 
	\end{proof}

	\begin{claim}
		\label{FC-5_24_2}
		No $5_{\geq 2}$-vertex has a $4_2$-neighbor. 
	\end{claim}
	\begin{proof}
		Assume to the contrary that $v_1$ is a $5_{\geq 2}$-vertex and $v_2$ is a $4_2$-vertex, with $v_1v_2 \in E(G)$. Let $v_3$ and $v_4$ be two $3$-neighbors of $v_1$, and let $v_5$ and $v_6$ be two $3$-neighbors of $v_2$. By Claim \ref{FC-k_{k-2}-3_1}, there are no edges between $\{v_3,v_4\}$ and $\{v_5,v_6\}$. Let $H=G[v_1,v_2, \ldots, v_6]$.  
		By Definition~\ref{def-(f_H,g_H)}, $(f_H,g_H)(v_1) = (5m,2m)$, $(f_H,g_H)(v_2)=(9m,4m)$,  $(f_H,g_H)(v_i) = (5m,3m)$ for  $i \in\{3,4\}$, and $(f_H,g_H)(v_j)=(7m,4m)$  for $j \in\{5,6\}$. 
		Then  \[\mathsf{Red}\bigl(
		\langle v_3\rangle,\ \langle v_4\rangle,\
		\langle v_2 \mid v_5, m\rangle,\ \langle v_5\rangle,\   
		\langle v_2 \mid v_6, m\rangle,\ \langle v_6\rangle,\
		\langle v_2\rangle,\ \langle v_1\rangle
		\bigr)\]  
		is legal for $(H,L_H, g_H)$ and deletes all vertices of $H$, which contradicts Lemma~\ref{key lemma}. 
	\end{proof} 
	
	\begin{claim}
		\label{FC-(5_2/6_3,3,3,4)}
		$G$ contains no $(5_{\ge 2},3,3,4)$- or $(6_{\ge 3},3,3,4)$-cycle. 
	\end{claim}
	\begin{proof}
		Suppose for a contradiction that $v_1v_2v_3v_4v_1$ is a cycle with $v_1$ a $5_{\ge 2}$- or $6_{\ge 3}$-vertex, $d(v_2)=d(v_3)=3$, and $d(v_4)=4$. Let $k=d(v_1)$, and let $v_5,\ldots,v_k$ be the remaining $3$-neighbors of $v_1$ distinct from $v_2$. By Claim~\ref{FC-star}, we have $v_3v_i\notin E(G)$ for each $i\in\{5,\ldots,k\}$. Let $H=G[v_1, v_2, \ldots, v_k]$. By Definition~\ref{def-(f_H,g_H)}, according to the values of $k \in \{5,6\}$, $(f_H,g_H)(v_1)=\big((10-k)m, (7-k)m\big)$, $(f_H,g_H)(v_2)=\big((14-k)m, 4m\big)$, $(f_H,g_H)(v_3)=(10m,4m)$, and $(f_H,g_H)(v_i)=\big((10-k)m,3m\big)$ for $i \in \{4, \ldots, k\}$.     
		Then 
		\[
		\mathsf{Red} \big(\langle v_5\rangle,\  \ldots,\ \langle v_k\rangle,\ \langle \{v_2, v_4\} \mid v_3, m^*\rangle,\ \langle v_3\rangle,\ \langle v_2\rangle,\ \langle v_4\rangle,\ \langle v_1\rangle \big)
		\] 
		is legal for $(H,L_H, g_H)$ and deletes all vertices of $H$, which contradicts Lemma~\ref{key lemma}. 
	\end{proof}
	
	\begin{claim}
		\label{FC-(5_2,3,4,4)}
		$G$ contains no $(5_{\ge 2},3,4,4)$-cycle. 
	\end{claim}
	\begin{proof}
		Assume to the contrary, $v_1v_2v_3v_4v_1$ is a $(5,3,4,4)$-cycle with $d(v_1)=5$, $d(v_2)=3$, $d(v_3)=d(v_4)=4$, and $v_1$ has an additional $3$-neighbor $v_5$. Let $H=G[v_1,v_2,\ldots,v_5]$. We determine $f_H$ and $g_H$ in two cases according to whether $v_3v_5 \in E(G)$.  
		\begin{itemize}
			\item If $v_3v_5 \notin E(G)$, then $(f_H,g_H)(v_i)=(4m,2m)$ for $i \in \{1,4\}$, $(f_H,g_H)(v_j)=(5m,3m)$ for $j \in \{3,5\}$, and 
			$(f_H,g_H)(v_2)=(8m,4m)$.
			\item If $v_3v_5 \in E(G)$, then $(f_H,g_H)(v_i)=(6m,3m)$ for $i \in \{1,4\}$,
			$(f_H,g_H)(v_i)=(10m,4m)$ for $i\in\{2,3,5\}$.
		\end{itemize}  
		
		For the first case, let $\Omega=\mathsf{Red}\big(\langle v_5\rangle,\ \langle \{v_1, v_3\} \mid v_2, m^*\rangle,\ \langle v_2\rangle,\ \langle v_1\rangle,\ \langle v_3\rangle,\ \langle v_4\rangle \big)$.    
		For the second case,  let 
		$\Omega=\mathsf{Red}\big(\langle \{v_1, v_3\} \mid v_2, m^*\rangle,\ \langle v_2\rangle,\ \langle \{v_1, v_3\} \mid v_5, m^*\rangle,\ \langle v_5\rangle,\ \langle v_1\rangle,\ \langle v_3\rangle,\ \langle v_4\rangle \big)$. 
		In  both cases, $\Omega$ is legal and deletes all vertices of $H$, which contradicts Lemma~\ref{key lemma}. 
	\end{proof}
	
	\begin{claim}
		\label{FC-(4,4,4,4_1/5_2)}
		$G$ contains no $(4,4,4,4_{\ge 1})$-cycle.
	\end{claim}
	
	\begin{proof}
		Assume to the contrary that $v_1v_2v_3v_4v_1$ is a $(4,4,4,4_{\geq1})$-cycle, where $v_1$ has a $3$-neighbor $v_5$. By Claim~\ref{FC-(4-,4-,4-,3-)}, $v_5v_3\notin E(G)$. Let $H=G[v_1,v_2,\ldots,v_5]$. By Definition~\ref{def-(f_H,g_H)}, $(f_H,g_H)(v_1)=(9m,4m)$, $(f_H,g_H)(v_3)=(5m,2m)$, $(f_H,g_H)(v_5)=(7m,4m)$, and $(f_H,g_H)(v_i)=(5m,3m)$ for $i\in\{2,4\}$.

		Let $L_H$ be an arbitrary $f_H$-list assignment of $H$. Whenever extra colors are deleted from a current list below, it suffices to color the resulting restricted list assignment. By Lemma~\ref{lem-three sets}, there exist $S\subseteq L_H(v_2)\setminus L_H(v_1)$, $T\subseteq L_H(v_4)\setminus L_H(v_1)$, and $R\subseteq L_H(v_2)\cap L_H(v_4)\cap L_H(v_1)$ such that $|S|+|T|+|R|=m$. Apply $\langle v_2,S\cup R\rangle$ and $\langle v_4,T\cup R\rangle$, and denote the resulting lists and demands by $L'$ and $g'$. By deleting extra colors, we may assume that $|L'(v_3)|=4m$. Next apply
		\[
		\mathsf{Red}\bigl(
		\langle v_2\mid v_3,|T|\rangle,\ \langle v_4\mid v_3,|S|\rangle,
		\langle v_1\mid v_5,m\rangle,\ \langle v_5\rangle,\ \langle v_1\rangle
		\bigr),
		\]
		and denote the resulting lists and demands by $L''$ and $g''$. Again by deleting extra colors, we may assume that $|L''(v_2)|=3m$. Then $\mathsf{Red}\bigl(\langle v_3\mid v_2,m\rangle,
		\langle v_2\rangle,\ \langle v_3\rangle,\ \langle v_4\rangle\bigr)$
		is legal for the remaining graph and deletes all its vertices, which contradicts Lemma~\ref{key lemma}.
	\end{proof}

	\begin{claim}
		\label{FC-(5_3/6_4,3,4,3)}
		$G$ contains no $(5_3,3,4,3)$- or $(6_4,3,4,3)$-cycle. 
	\end{claim} 
	\begin{proof}
		Assume to the contrary, $v_1v_2v_3v_4v_1$ is a cycle such that $v_1$ is either a $5_3$-vertex or a $6_4$-vertex, $d(v_2) = d(v_4)=3$ and $d(v_3)=4$. Let $k=d(v_1)$ and $v_5, \ldots, v_k$ be $3$-neighbors of $v_1$ other than $v_2$ and $v_4$. By Claim \ref{FC-star}, $v_3v_5,\ldots, v_3v_k\notin E(G)$. Let $H=G[v_1, v_2, \ldots, v_k]$.   
		By Definition~\ref{def-(f_H,g_H)}, according to the cases of $k \in \{5,6\}$, $(f_H,g_H)(v_1)=\big((11-k)m, (8-k)m\big)$, $(f_H,g_H)(v_i)=\big((15-k)m, 4m\big)$ for $i \in \{2, 4\}$, $(f_H,g_H)(v_3)=(7m,4m)$, and $(f_H,g_H)(v_j)=\big((11-k)m, 3m\big)$ for $j \in \{5, \ldots, k\}$.  
		Then
		\[
		\mathsf{Red}\bigl(
		\langle v_5\rangle,\ \langle v_6\rangle,\ \ldots,\ \langle v_k\rangle,\
		\langle \{v_1, v_3\} \mid v_2, m^*\rangle,\
		\langle v_2\rangle,\ \langle v_4\mid v_3,m\rangle,\
		\langle v_3\rangle,\ \langle v_4\rangle,\  \langle v_1\rangle \bigr) 
		\] is legal for $(H,L_H,g_H)$, and delete all the vertices of $H$, which contradicts Lemma~\ref{key lemma}. 
	\end{proof} 
	
	\begin{claim}
		\label{FC-(5_1,4,3,4)}
		$G$ contains no $(5_{\ge 1},4,3,4)$-cycle. 
	\end{claim}
	\begin{proof}
		Assume to the contrary that $v_1v_2v_3v_4v_1$ is a $4$-cycle with 
		$d(v_1)=5$, $d(v_2)=d(v_4)=4$, and $d(v_3)=3$, 
		and that $v_1$ has a $3$-neighbor $v_5$. 
		Let $H=G[\{v_1,v_2,v_3,v_4,v_5\}]$. 
		We distinguish two cases according to whether $v_3v_5\in E(H)$.

		\medskip
		\noindent 
		{\bf Case 1.} $v_3v_5 \in E(H)$.
		\medskip
		
		By Definition~\ref{def-(f_H,g_H)}, $(f_H,g_H)(v_1)=(5m,2m)$, $(f_H,g_H)(v_3)=(13m,4m)$, $(f_H,g_H)(v_i)=(5m,3m)$ for $i \in \{2,4\}$, and $(f_H,g_H)(v_5)=(9m,4m)$.  
		Clearly,
		\[\mathsf{Red}\bigl(
		\langle \{v_2, v_5\} \mid v_3, m^* \rangle,\
		\langle v_3\rangle,\  \langle v_4\rangle,\ \langle v_5\rangle,\ \langle v_2\rangle,\ \langle v_1\rangle
		\bigr)\] is legal for $(H,L_H,g_H)$ and delete all the vertices of $H$, which contradicts Lemma~\ref{key lemma}.  
		
		\medskip
		\noindent 
		{\bf Case 2.} $v_3v_5 \notin E(H)$.
		\medskip
		
		By the minimality, $G-v_3$ has an $(L,4m)$-coloring $\phi$. Let $L_0$ be a list assignment of $V(H)$ defined as $L_0(u) = L(u) \setminus \{\phi(v) \mid v \in N_{G - V(H)}(u)\}$ for each $u \in V(H)$. By deleting extra colors, we assume that $|L_0(v_3)|=11m$. Since $\phi$ is an $(L,4m)$-coloring of $G-v_3$, for each $i\in\{2,4,5\}$, there exists an $m$-set $A_i \subseteq L_0(v_i)\setminus L_0(v_1)$, and an $m$-set $B_i \subseteq L_0(v_1)\setminus L_0(v_i)$. Let $X=L_0(v_2) \setminus (B_4 \cup B_5)$, $Y = L_0(v_4) \setminus (B_2 \cup B_5)$ and let $Z$ be a $9m$-subset of $L_0(v_3) \setminus (A_2 \cup A_4)$. It is clear that $|X| \geq 5m$, $|Y| \geq 5m$. Thus $|X|+|Y| \geq |Z|+m$, hence 
		by Lemma~\ref{lem-three sets}, there exists $S \subseteq X \setminus Z$,
		$T \subseteq Y \setminus Z$, and $R\subseteq X \cap Y \cap Z$ such that $|S|+|T|+|R|=m$.  
		Then  
		\[ 
		\mathsf{Red}\bigl(
		\langle v_2, S\cup R\rangle,\ \langle v_4, T \cup R \rangle,\
		\langle v_3\rangle,\  \langle v_1, B_2\cup B_4 \cup B_5\rangle,\ \langle v_2\rangle,\ \langle v_4\rangle,\ \langle v_5\rangle,\ \langle v_1\rangle
		\bigr)
		\] 
		is legal for $(H,L_0,4m)$ and delete all the vertices of $H$, which contradicts Lemma~\ref{key lemma}.  
	\end{proof}
	
	\begin{claim}
		\label{FC-(5_2,4,4_1,4)}
		$G$ contains no $(5_{\geq2},4,4_{\geq1},4)$-cycle.
	\end{claim}
	\begin{proof}
		Assume to the contrary that $v_1v_2v_3v_4v_1$ is such a cycle, where $d(v_1)=5$, $d(v_i)=4$ for $i\in\{2,3,4\}$, $v_1$ has two distinct $3$-neighbors $v_5,v_6$, and $v_3$ has a $3$-neighbor $v_7$. By Claim~\ref{FC-(5_2,3,4,4)}, $v_7\notin\{v_5,v_6\}$ and $v_3v_5,v_3v_6\notin E(G)$. By Claim~\ref{FC-star}, at most one of $v_5v_7$ and $v_6v_7$ lies in $E(G)$. Let $H=G[v_1,v_2,\ldots,v_7]$. Whenever extra colors are deleted from a current list below, it suffices to color the resulting restricted list assignment.

		By symmetry, if one of $v_5v_7$ and $v_6v_7$ is an edge, we may assume that $v_6v_7\in E(G)$. In this case, $(f_H,g_H)(v_i)=(7m,4m)$ for $i\in\{2,4,5\}$ and $(f_H,g_H)(v_j)=(11m,4m)$ for $j\in\{1,3,6,7\}$. Let $L_H$ be an arbitrary $f_H$-list assignment. First apply $\mathsf{Red}(\langle v_1\mid v_5,m\rangle,\langle v_5\rangle)$. By deleting extra colors from the current list of $v_6$, we may assume that it has size $10m$. Now apply $\mathsf{Red}(\langle v_7\mid v_6,m\rangle,\langle v_6\rangle,\langle v_7\rangle)$, and denote the resulting lists and demands on the remaining cycle by $L'$ and $g'$. By deleting extra colors, we may assume that
		\[
		(|L'(v_1)|,|L'(v_2)|,|L'(v_3)|,|L'(v_4)|)=(10m,6m,10m,6m)
		\]
		and $(g'(v_1),g'(v_2),g'(v_3),g'(v_4))=(3m,4m,4m,4m)$. Then we have
		\[\mathsf{Red}\bigl(
		\langle \{v_2, v_4\} \mid v_1, m^*\rangle,\ \langle v_1\rangle,\ \langle \{v_2, v_4\} \mid v_3, m^*\rangle,\ \langle v_3\rangle,\ \langle v_2\rangle,\ \langle v_4\rangle \bigr)
		\]
		is legal for $(H,L_H,g_H)$ and deletes the remaining vertices $v_1,v_2,v_3,v_4$.

		Now assume that $v_5v_7,v_6v_7\notin E(G)$. Then $(f_H,g_H)(v_i)=(7m,4m)$ for $i\in\{2,4,5,6,7\}$ and $(f_H,g_H)(v_j)=(11m,4m)$ for $j\in\{1,3\}$. Let $L_H$ be an arbitrary $f_H$-list assignment. By Lemma~\ref{lem-three sets}, we may apply
		\[
		\mathsf{Red}\Bigl(
		\begin{array}{l}
		\langle \{v_2,v_4\}\mid v_3, m^*\rangle,\ \langle \{v_2,v_4\}\mid v_1, m^*\rangle,
		\langle v_3\mid v_7,m\rangle,\ \langle v_7\rangle,\ \langle v_3\rangle,\\
		\langle v_1\mid v_5,m\rangle,\ \langle v_5\rangle,\ 
		\langle v_1\mid v_6,m\rangle,\ \langle v_6\rangle,\ \langle v_1\rangle,\ 
		\langle v_2\rangle,\ \langle v_4\rangle
		\end{array}
		\Bigr)
		\]
		is legal for $(H,L_H,g_H)$ and deletes all vertices of $H$. In both cases, we obtain a contradiction to Lemma~\ref{key lemma}.
	\end{proof}

	\begin{claim}
		\label{FC-6-(6_3,4,3,4)}
		$G$ contains no $(6_{\geq 3},4,3,4)$-cycle.
	\end{claim}
	\begin{proof}
		Assume to the contrary, $v_1v_2v_3v_4v_1$ is such a cycle where $v_1$ is the $6_{\geq 3}$-vertex, and let $v_5,v_6,v_7$ be three $3$-neighbors of $v_1$. By Claim~\ref{FC-(5_2/6_3,3,3,4)}, we have $v_3v_i \notin E(G)$  for $i \in \{5,6,7\}$. Let $H=G[v_1, v_2, \ldots, v_7]$.  
		By Definition~\ref{def-(f_H,g_H)}, $(f_H,g_H)(v_j)=(11m,4m)$ for $j\in\{1,3\}$ and $(f_H,g_H)(v_i)=(7m,4m)$ for $i\in\{2,4,5,6,7\}$. Then
		\[\mathsf{Red}\Bigl(
		\begin{array}{l} 
			\langle \{v_2,v_4\} \mid v_3, m^*\rangle,\ \langle v_3\rangle,\ \langle \{v_2,v_4\} \mid v_1, m^*\rangle,\ \langle v_1 \mid v_5, m \rangle,\\ \langle v_5\rangle,\ \langle v_1 \mid v_6, m \rangle,\ \langle v_6 \rangle,\ \langle v_1 \mid v_7, m \rangle,\ \langle v_7\rangle,\ \langle v_1\rangle,\  \langle v_2\rangle,\ \langle v_4\rangle 
		\end{array}  
		\Bigr)
		\] is legal for $(H,L_H,g_H)$ and deletes all vertices of $H$, which contradicts Lemma~\ref{key lemma}. 
	\end{proof}
	
	\begin{claim}
		\label{FC-(5,3,3,4)-4_1}
		If a $5$-vertex lies on a $(5,3,3,4)$-cycle, 
		then it has no $4_{\ge 1}$-neighbors other than those on the cycle. 
	\end{claim}
	\begin{proof}
		Assume to the contrary, $v_1v_2v_3v_4v_1$ is a $(5,3,3,4)$-cycle with $d(v_1)=5$, $d(v_2)=d(v_3)=3$ and $d(v_4)=4$ and $v_1$ has a $4$-neighbor $v_5$ distinct from $v_4$. Moreover, let $v_6$ be a $3$-neighbor of $v_5$, chosen so that 
		$v_6 \neq v_3$ unless $v_3$ is the unique $3$-neighbor of $v_5$. If $v_6\neq v_3$, then by Claim~\ref{FC-star}, $v_2v_6, v_3v_6 \notin E(G)$, and by Claim~\ref{FC-k_{k-2}-3_1}, $v_5v_3,v_6v_4 \notin E(G)$. We consider two cases to determine $(f_H,g_H)$ by Definition~\ref{def-(f_H,g_H)}.
		\begin{itemize}
			\item If $v_6 \neq v_3$, then let $H=G[v_1,v_2,\ldots,v_6]$. Then $(f_H,g_H)(v_2)=(9m,4m)$, $(f_H,g_H)(v_i)=(4m,2m)$ for $i\in\{1,5\}$,
			$(f_H,g_H)(v_3)=(10m,4m)$, and $(f_H,g_H)(v_j)=(5m,3m)$ for $j\in\{4,6\}$.
			\item If $v_6 = v_3$, then let $H=G[v_1,v_2,\ldots,v_5]$. Then  $(f_H,g_H)(v_1)=(5m,2m)$, 
			$(f_H,g_H)(v_2)=(9m,4m)$,  $(f_H,g_H)(v_3)=(13m,4m)$, and $(f_H,g_H)(v_i)=(5m,3m)$ for $i\in\{4,5\}$.
		\end{itemize}
		
		For the former case, let $\Omega=\mathsf{Red}\big(\langle v_6\rangle,\ \langle v_5\rangle,\ \langle \{v_2, v_4\} \mid v_3, m^*\rangle,\ \langle v_3\rangle,\ \langle v_2\rangle,\ \langle v_4\rangle,\ \langle v_1\rangle \big)$.   
		For the latter case, let $\Omega=\mathsf{Red}(\langle \{v_2, v_4\} \mid v_3, m^*\rangle,\ \langle v_3\rangle,\ \langle v_5\rangle, ~\langle v_2\rangle,\ \langle v_4\rangle,\ \langle v_1\rangle)$. 
		In either case, $\Omega$ is legal for $(H,L_H,g_H)$ and delete all the vertices of $H$, which contradicts Lemma~\ref{key lemma}. 
	\end{proof}

	\begin{claim}
		\label{FC-334~43}
		$G$ contains no $(3,3,4,4,3)$- or $(3,3,4,4,4,3)$-path. 
	\end{claim} 
	\begin{proof}
		Suppose to the contrary, $H=v_1v_2\dots v_k$ is such a path, where $k \in \{5,6\}$, $d(v_1)=d(v_2)=d(v_k)=3$, and $d(v_i)=4$ for $3\le i\le k-1$. By Claims \ref{FC-star} and \ref{FC-k_{k-2}-3_1}, $G$ contains neither a $(3,3,3)$-path nor a $(3,3,4,3)$-path. By Claim~\ref{FC-(4-,4-,4-,3-)},  $G$ contains no $(4,4,3,3)$-cycle. Consequently, $H$ is an induced path. By Definition~\ref{def-(f_H,g_H)}, $(f_H,g_H)(v_1)=(7m,4m)$, $(f_H,g_H)(v_2)=(10m,4m)$, 
		$(f_H,g_H)(v_i)=(5m,3m)$ for $i\in\{3,k\}$ and $(f_H,g_H)(v_j)=((10-k)m,2m)$ for $4\leq j\leq k-1$.  
		
		If $k=5$, then let  $\Omega=\mathsf{Red}\bigl(
		\langle v_5\rangle,\ \langle v_4\rangle,\
		\langle v_2 \mid v_1,m\rangle,\   \langle v_1\rangle,\     
		\langle v_2\rangle,\ 
		\langle v_3\rangle
		\bigr)$. 
		If $k=6$, then let $
		\Omega=\mathsf{Red}\bigl(
		\langle v_6\rangle,\  \langle v_5\rangle,\ \langle v_3 \mid v_4, m\rangle,\ \langle v_4\rangle,~\langle v_2 \mid v_1,m\rangle,\   \langle v_1\rangle,\ \langle v_2\rangle,\   
		\langle v_3\rangle
		\bigr).$ 
		In either case, $\Omega$ is legal for $(H, L_H, g_H)$ and deletes all vertices of $H$, which contradicts Lemma~\ref{key lemma}.   
	\end{proof}

	\begin{claim} 
		\label{FC-3434~43} 
		$G$ contains no $(3,4,3,4,3)$- or $(3,4,3,4,4,3)$-path. 
	\end{claim}
	
	\begin{proof}
		Assume to the contrary, $H=v_1v_2\cdots v_k$ is such a path,  
		where $k\in \{5,6\}$, $d(v_i)=3$ for $i \in \{1,3,k\}$ and $d(v_j)=4$ for $j \in \{2,4,\ldots,k-1\}$. By Claims~\ref{FC-k_{k-2}-3_1},~\ref{FC-(4-,4-,4-,3-)} and \ref{FC-5_3/4_24_1}, $H$ is an induced path. 
		By Definition~\ref{def-(f_H,g_H)},
		$(f_H,g_H)(v_i)=(6m,3m)$ for $i\in\{1,2\}$, $(f_H,g_H)(v_3)=(9m,4m)$. If $k=5$, then $(f_H,g_H)(v_j)=(6m,3m)$ for $j \in \{4,5\}$. If $k=6$, then $(f_H,g_H)(v_j)=(5m,3m)$ for $j \in \{4,6\}$ and $(f_H,g_H)(v_5)=(5m,2m)$. 
		Then
		$$\Omega=\mathsf{Red}\bigl(
		\langle v_1\rangle,\  \langle v_{k} \rangle,\
		\langle v_{k-1}\rangle,\ \dots,~ \langle v_5\rangle,\ \langle v_3\mid v_2, m \rangle,\ \langle v_2 \rangle,\  \langle v_3\rangle,\ \langle v_4\rangle \bigr)$$
		is legal for $(H,L_H, g_H)$ and deletes all vertices of $H$, which contradicts Lemma~\ref{key lemma}.  
	\end{proof}

	\begin{claim}
		\label{FC-4_14_14_1}
		$G$ contains no $(4_{\ge 1},4_{\ge 1},4_{\ge 1})$-path. 
	\end{claim}
	\begin{proof}
		Assume to the contrary, $v_1v_2v_3$ is such a path, where
		$v_1$, $v_2$, and $v_3$ are $4_{\ge 1}$-vertices.
		Let $v_4$ be a $3$-neighbor of $v_1$, $v_5$ be a $3$-neighbor of $v_3$, and
		$v_6$ be a $3$-neighbor of $v_2$. Since $G$ is triangle-free, and by Claim~\ref{FC-(4-,4-,4-,3-)}, $v_4,v_5,v_6$ are all distinct, moreover, none of $v_4v_6, v_5v_6, v_3v_4, v_1v_5$ is in $E(G)$. By Claim \ref{FC-334~43}, $v_4v_5\notin E(G)$. Let $H = G[v_1, v_2, \ldots, v_6]$, and by Definition~\ref{def-(f_H,g_H)}, 
		$(f_H,g_H)(v_2)=(9m,4m)$, $(f_H,g_H)(v_6)=(7m,4m)$, and
		$(f_H,g_H)(v_i)=(6m,3m)$ for $i\in\{1,3,4,5\}$.  
		Then
		$$
		\mathsf{Red}\bigl(
		\langle v_4\rangle,\ 	\langle v_5\rangle,\ 
		\langle v_2\mid v_6, m\rangle,\ \langle v_6\rangle,\ 
		\langle v_2\mid v_3, m\rangle,\ \langle v_3\rangle,\    
		\langle v_2\rangle,\ \langle v_1\rangle
		\bigr)
		$$
		is legal for $(H,L_H, g_H)$ and deletes all vertices of $H$, which contradicts Lemma~\ref{key lemma}.  
	\end{proof} 
	
	\begin{claim}
		\label{FC-3443443}
		$G$ contains no $(3,4,4,3,4,4,3)$-path. 
	\end{claim}
	\begin{proof}
		Assume to the contrary, $H=v_1v_2\cdots v_7$ is such a path,  
		where $d(v_j)=3$ for $j \in \{1,4,7\}$ and $d(v_i)=4$ for $i \in \{2,3,5,6\}$. By Claims~\ref{FC-k_{k-2}-3_1},~\ref{FC-(4-,4-,4-,3-)},~\ref{FC-334~43} and \ref{FC-4_14_14_1}, $H$ is induced. By Definition~\ref{def-(f_H,g_H)}, $(f_H,g_H)(v_i)=(5m,3m)$ for $i\in\{1,3,5,7\}$, $(f_H,g_H)(v_j)=(5m,2m)$ for $j\in\{2,6\}$, and $(f_H,g_H)(v_4)=(9m,4m)$. 
		Then  $$\mathsf{Red}\bigl(
		\langle v_1\rangle,\ \langle v_2\rangle,\
		\langle v_7\rangle,\ \langle v_6\rangle,\
		\langle v_4\mid v_3, 2m\rangle,\ \langle v_3\rangle,\  
		\langle v_4\rangle,\ \langle v_5\rangle
		\bigr)$$
		is legal for $(H,L_H, g_H)$ and deletes all vertices of $H$, which contradicts Lemma~\ref{key lemma}.   
	\end{proof}

	\begin{claim}
		\label{FC-4_2-4/5_1-4_1} 
		$G$ contains no $(4_2,4,4_{\ge 1})$- or $(4_2,5_{\ge 1},4_{\ge 1})$-path. 
	\end{claim} 
	\begin{proof}
		Assume to the contrary, $v_1v_4v_5$ is such a path, where $v_1$ is a $4_2$-vertex with two distinct $3$-neighbors $v_2$ and $v_3$, $v_5$ is a $4_{\geq 1}$-vertex with a $3$-neighbor $v_6$, and $v_4$ is either a $4$-vertex or a $5_{\geq 1}$-vertex.  
		Let $v_7$ be a $3$-neighbor of $v_4$ if $d(v_4)=5$. Since $G$ contains no triangle, $v_7 \neq v_i$ for $i \in \{2,3,6\}$. By Claims~\ref{FC-(4-,4-,4-,3-)} and \ref{FC-(5_1,4,3,4)}, $v_6 \neq v_2, v_6 \neq v_3$. Let $k = d(v_4) \in \{4,5\}$, and $H = G[v_1, v_2, \ldots, v_{k+2}]$. By Claim~\ref{FC-k_{k-2}-3_1}, $v_2v_6,v_3v_6 \notin E(H)$, and moreover, if $k=5$, then $v_7v_2, v_7v_3 \notin E(H)$. By Claim~\ref{FC-3434~43},  $v_5v_2, v_5v_3 \notin E(H)$. By Claim~\ref{FC-star}, $v_1v_6 \notin E(H)$. If $k=5$, then by Claim~\ref{FC-(5,3,3,4)-4_1}, $v_6v_7\notin E(H)$. Thus
		by Definition~\ref{def-(f_H,g_H)}, $(f_H,g_H)(v_i)=(7m,4m)$ for $i \in \{2,3\}$, $(f_H,g_H)(v_6)=(5m,3m)$. If $k=4$, then $(f_H,g_H)(v_1)=(10m,4m)$, $(f_H,g_H)(v_4)=(5m,3m)$, and  $(f_H,g_H)(v_5)=(5m,2m)$. If $k=5$, then $(f_H,g_H)(v_1)=(9m,4m)$, $(f_H,g_H)(v_i)=(4m,2m)$ for $i \in \{4,5\}$, and $(f_H,g_H)(v_7)=(5m,3m)$. Then 
		\[ 
		\mathsf{Red}\bigl(
		\langle v_{k+2}\rangle,\ \langle v_{k+1}\rangle,\ \ldots,\ \langle v_5\rangle,\   
		\langle v_1\mid v_2, m\rangle,\ \langle v_2\rangle,\  \langle v_1\mid v_3, m\rangle,\ \langle v_3\rangle,\ \langle v_1\rangle,\ \langle v_4\rangle
		\bigr)
		\]
		is legal for $(H,L_H, g_H)$ and deletes all vertices of $H$, a contradiction.  
	\end{proof}
	
	\begin{claim}
		\label{FC-3_15_24_1}
		$G$ contains no $(3_1,5_{\ge 2},4_{\ge 1})$-path. 
	\end{claim}
	\begin{proof} 
		Assume to the contrary, $v_1v_2v_3$ is such a path, where $v_1$ is a $3_1$-vertex,
		$v_2$ is a $5_{\ge 2}$-vertex, and $v_3$ is a $4_{\ge 1}$-vertex.
		Let $v_4$ be the $3$-neighbor of $v_1$, $v_5$ be a $3$-neighbor of $v_2$ distinct from $v_1$, and $v_6$ be a $3$-neighbor of $v_3$. By Claim~\ref{FC-(5_2/6_3,3,3,4)}, $v_6\neq v_4$ and $v_5v_6 \notin E(G)$.
		By Claim~\ref{FC-k_{k-2}-3_1}, $v_3v_4 \notin E(G)$. By Claim~\ref{FC-star}, $v_4v_5, v_4v_6, v_1v_6\notin E(G)$. Let $H=G[v_1, v_2, \ldots, v_6]$. By
		Definition~\ref{def-(f_H,g_H)}, $(f_H,g_H)(v_1)=(9m,4m)$, $(f_H,g_H)(v_i)=(5m,3m)$ for $i\in\{5,6\}$, $(f_H,g_H)(v_4)=(7m,4m)$, and $(f_H,g_H)(v_j)=(4m,2m)$ for $j\in\{2,3\}$. 
		Then
		$$
		\mathsf{Red}\bigl(
		\langle v_6\rangle,\ \langle v_3\rangle,\ \langle v_5\rangle,\ 
		\langle v_1 \mid v_4, m\rangle,\ \langle v_4\rangle,\ \langle v_1\rangle,\ \langle v_2\rangle 
		\bigr)
		$$
		is legal for $(H,L_H, g_H)$ and deletes all vertices of $H$, which contradicts Lemma~\ref{key lemma}.   
	\end{proof} 
	
	\begin{claim}
		\label{FC-(5,3,4,3)-4_1}
		If a $5$-vertex lies on a $(5,3,4,3)$-cycle, then it has no $4_{\ge 1}$-neighbors. 
	\end{claim}
	\begin{proof}
		Assume to the contrary, $v_1v_2v_3v_4v_1$ is a $(5,3,4,3)$-cycle, where $d(v_1)=5$,
		$d(v_2)=d(v_4)=3$, $d(v_3)=4$, and $v_1$ has a $4_{\ge 1}$-neighbor $v_5$. Let $v_6$ be a $3$-neighbor of $v_5$. By Claim~\ref{FC-5_3/4_24_1}, $v_5v_3 \notin E(G)$. By Claim~\ref{FC-star}, $v_6v_3 \notin E(G)$. By Claim~\ref{FC-k_{k-2}-3_1}, $v_6v_2, v_6v_4 \notin E(G)$. Let $H=G[v_1, v_2, \ldots, v_6]$.
		By Definition~\ref{def-(f_H,g_H)}, $(f_H,g_H)(v_3)=(7m,4m)$, $(f_H,g_H)(v_i)=(10m,4m)$ for $i\in\{2,4\}$, $(f_H,g_H)(v_5)=(5m,2m)$, and $(f_H,g_H)(v_j)=(5m,3m)$ for $j\in\{1,6\}$. 
		Then
		$$
		\mathsf{Red}\bigl(\langle v_6\rangle,\ \langle v_5\rangle,\
		\langle \{v_1, v_3\} \mid v_2, m^*\rangle,\
		\langle v_2\rangle,\ \langle v_4\mid v_3,m\rangle,\
		\langle v_3\rangle,\ \langle v_4\rangle,\ \langle v_1\rangle
		\bigr)$$ is legal for $(H,L_H, g_H)$ and deletes all vertices of $H$, which contradicts Lemma~\ref{key lemma}. 
	\end{proof}
	
	\begin{claim}
		\label{FC-(5_2,3_1,4,5_1)}
		If $v_1v_2v_3v_4v_1$ is a $(5,3,4,5_{\ge 1})$-cycle with 
		$d(v_1)=5$, $d(v_2)=3$, $d(v_3)=4$, and $v_4$ a $5_{\ge 1}$-vertex, 
		then $v_1$ has no $3$-neighbors other than $v_2$. 
	\end{claim}
	\begin{proof}
		Assume to the contrary that $v_5$ is a $3$-neighbor of $v_1$ distinct from $v_2$. Let $v_6$ be a $3$-neighbor of $v_4$. 	Let $H=G[v_1, v_2, \ldots, v_6]$. 
		
		\medskip
		\noindent 
		{\bf Case 1.} At least one of $v_3v_5$, $v_2v_6$, and $v_5v_6$ lies in $E(H)$. 
		\medskip
		
		By Claim~\ref{FC-k_{k-2}-3_1} and the fact that $G$ contains no triangle, 
		exactly one of $v_3v_5$, $v_2v_6$, and $v_5v_6$ lies in $E(H)$. We consider three cases to determine $(f_H, g_H)$ by   Definition~\ref{def-(f_H,g_H)}.
		\begin{itemize}
			\item If $v_3v_5 \in E(H)$, then $(f_H,g_H)(v_i)=(5m,3m)$ for $i\in\{1,6\}$, $(f_H,g_H)(v_i)=(10m,4m)$ for $i\in\{2,5\}$, $(f_H,g_H)(v_3)=(9m,4m)$, and $(f_H,g_H)(v_4)=(5m,2m)$.
			
			\item If $v_2v_6 \in E(H)$, then  $(f_H,g_H)(v_i)=(4m,2m)$ for $i\in\{1,4\}$, $(f_H,g_H)(v_j)=(5m,3m)$ for $j\in\{3,5\}$, $(f_H,g_H)(v_2)=(12m,4m)$, and $(f_H,g_H)(v_6)=(9m,4m)$. 
			
			\item If $v_5v_6 \in E(H)$, then $(f_H,g_H)(v_i)=(5m,3m)$ for $i\in\{1,3\}$, $(f_H,g_H)(v_j)=(9m,4m)$ for $j \in \{2,6\}$, $(f_H,g_H)(v_4)=(5m,2m)$, and $(f_H,g_H)(v_5)=(10m,4m)$. 
		\end{itemize}
		
		For the first case, let 
		\[
		\Omega =
		\mathsf{Red}\bigl(
		\langle v_6\rangle,\  
		\langle \{v_1, v_3\} \mid v_2, m^*\rangle,\ \langle v_2\rangle,\ \langle \{v_1, v_3\} \mid v_5, m^*\rangle,\ \langle v_5\rangle,\ \langle v_1\rangle,\ \langle v_3\rangle,\ \langle v_4\rangle \bigr).
		\] 
		
		For the second case, let $
		\Omega = \mathsf{Red}\bigl(
		\langle v_5\rangle,\
		\langle \{v_1, v_6\} \mid v_2, m^*\rangle,\ \langle v_2\rangle,\ \langle v_1\rangle,\ \langle v_3\rangle,\ \langle v_6\rangle,\ \langle v_4\rangle \bigr)
		$.
		
		For the third case, let 
		$$
		\Omega = \mathsf{Red}\bigl(
		\langle \{v_1, v_3\} \mid v_2, m^*\rangle,\
		\langle v_2\rangle,\ \langle v_3\rangle,\ \langle \{v_1, v_6\} \mid v_5, m^*\rangle,\ \langle v_5\rangle,\ \langle v_6\rangle,\ \langle v_1\rangle,\ \langle v_4\rangle
		\bigr).$$
		
		In each case, $\Omega$  is legal  and deletes all vertices of $H$,  which contradicts Lemma~\ref{key lemma}.  
		
		\medskip
		\noindent 
		{\bf Case 2.} None of $v_3v_5$, $v_2v_6$, and $v_5v_6$ lies in $E(H)$. 
		\medskip
		
		By minimality, $G-v_2$ has an $(L,4m)$-coloring $\phi$. Let $L_0$ be a list assignment of $V(H)$ defined as $L_0(u) = L(u) \setminus \{\phi(v) \mid v \in N_{G - V(H)}(u)\}$ for each $u \in V(H)$. By deleting extra colors not used by $\phi(v)$, we may assume that $|L_0(v_2)|=11m$ and $|L_0(v_i)|=7m$ for $i\in [6] \setminus \{2\}$. 
		Since $\phi$ is an $(L,4m)$-coloring of $G-v_2$, there exist $m$-sets
		$A_i\subseteq L_0(v_1)\setminus L_0(v_i)$ for $4\le i\le 5$,
		$m$-sets $B_j\subseteq L_0(v_4)\setminus L_0(v_j)$ for $j\in\{1,3,6\}$,
		and an $m$-set $C_3\subseteq L_0(v_3)\setminus L_0(v_4)$. 
		Let $L'_0(v_2)$ be an $8m$-subset of $L_0(v_2)\setminus(A_4\cup A_5\cup C_3)$, $L'_0(v_1)=L_0(v_1)\setminus(B_3\cup B_6)$ and $L'_0(v_3)=L_0(v_3)\setminus(B_1\cup B_6)$. Then $|L'_0(v_i)|\geq 5m$ for $i\in \{1,3\}$. 
		Since $|L'_0(v_1)|+|L'_0(v_3)| \geq |L'_0(v_2)|+m$, by Lemma~\ref{lem-three sets}, there exist $S\subseteq L'_0(v_1)\setminus L'_0(v_2)$, $T\subseteq L'_0(v_3)\setminus L'_0(v_2)$, and $R\subseteq L'_0(v_1)\cap L'_0(v_3) \cap L'_0(v_2)$ such that $|S|+|T|+|R|=m$.  
		Then
		\[
		\mathsf{Red}\bigl(
		\langle v_1, A_4\cup A_5\cup R\cup S\rangle,\ \langle v_5\rangle,\
		\langle v_3,C_3\cup R\cup T\rangle,\ 
		\langle v_2\rangle,\ \langle v_4,B_1\cup B_3\cup B_6\rangle,\ 
		\langle v_6\rangle,\
		\langle v_1\rangle,\   \langle v_3\rangle,\ \langle v_4\rangle
		\bigr)\] 
		is legal for $(H,L_0, 4m)$ and deletes all vertices of $H$, which contradicts Lemma~\ref{key lemma}.  
	\end{proof}

	\begin{claim}
		\label{FC-(5,4,3,4)-4_2}
		Suppose that $v_1v_2v_3v_4v_1$ is a $(5,4,3,4)$-cycle with $d(v_1)=5$. 
		Then $v_1$ has no $4_2$-neighbors. 
	\end{claim}
	 
	\begin{proof}
		Assume to the contrary, $v_1$ has a $4_2$-neighbor $v_5$ with two $3$-neighbors $v_6$ and $v_7$.
		
		\medskip
		\noindent 
		{\bf Case 1.} $v_6\neq v_3$ and $v_7\neq v_3$.
		\medskip
		
		Let $H=G[v_1, v_2, \ldots, v_7]$. By Claim~\ref{FC-3434~43}, $v_6v_2, v_6v_4, v_7v_2,v_7v_4 \notin E(H)$. By Claim~\ref{FC-star}, $v_3v_5\notin E(H)$. By Claim~\ref{FC-k_{k-2}-3_1}, $v_3v_6, v_3v_7 \notin E(H)$. Thus by Definition~\ref{def-(f_H,g_H)}, $(f_H,g_H)(v_1)=(5m,2m)$, $(f_H,g_H)(v_i)=(5m,3m)$ for $i \in \{2,4\}$,  $(f_H,g_H)(v_j)=(9m,4m)$ for $j \in \{3,5\}$, and $(f_H,g_H)(v_k)=(7m,4m)$ for $k \in \{6,7\}$.   
		Then
		$$\mathsf{Red}\bigl(  
		\langle \{v_2, v_4\} \mid v_3, m^*\rangle,\
		\langle v_3\rangle,\ \langle v_2\rangle,\ \langle v_4 \rangle,\ 
		\langle v_5 \mid v_6, m \rangle,\ \langle v_6 \rangle,\
		\langle v_5 \mid v_7, m \rangle,\ 
		\langle v_7 \rangle,\ \langle v_5 \rangle,\ 
		\langle v_1\rangle
		\bigr)$$  is legal  for $(H, L_H, g_H)$ and deletes all vertices of $H$,  which contradicts Lemma~\ref{key lemma}.
		
		\medskip
		\noindent 
		{\bf Case 2.} $v_6=v_3$ or $v_7=v_3$, say $v_7=v_3$.
		\medskip
		
		Let $H=G[v_1, v_2, \ldots, v_6]$. By Claim~\ref{FC-(4-,4-,4-,3-)}, $v_6v_2, v_6v_4 \notin E(H)$. Thus by Definition~\ref{def-(f_H,g_H)}, $(f_H,g_H)(v_1)=(5m,2m)$, $(f_H,g_H)(v_i)=(5m,3m)$ for $i \in \{2,4\}$,  $(f_H,g_H)(v_3)=(13m,4m)$, $(f_H,g_H)(v_5)=(9m,4m)$, and $(f_H,g_H)(v_6)=(7m,4m)$.   
		Then
		$$\mathsf{Red}\bigl( 
		\langle \{v_4, v_5\} \mid v_3, m^*\rangle,\  
		\langle v_3\rangle,\ \langle v_4 \rangle,\ \langle v_2\rangle,\ 
		\langle v_5\mid v_6,m\rangle,\ \langle v_6 \rangle,\ 
		\langle v_5\rangle,\ \langle v_1 \rangle
		\bigr)$$  is legal for $(H, L_H, g_H)$ and deletes all vertices of $H$,  which contradicts Lemma~\ref{key lemma}.  
	\end{proof}
	
	\begin{claim}
		\label{FC-adjcent-(5,3,3,4)&(5,3,4,4)}
		If $v_1v_2v_3v_4v_1$ is a $(5,3,3,4)$-cycle with 
		$d(v_1)=5$, $d(v_2)=d(v_3)=3$, and $d(v_4)=4$, 
		then the edge $v_1v_2$ is not contained in any $(5,3,4,4)$-cycle. 
	\end{claim}
	\begin{proof}
		Assume to the contrary, $v_1v_2v_5v_6v_1$ is a $(5,3,4,4)$-cycle in $G$, where $d(v_5)=d(v_6)=4$. By Claim~\ref{FC-(4-,4-,4-,3-)}, $v_3v_6, v_4v_5 \notin E(G)$. Let $H=G[v_1, v_2, \ldots, v_6]$.  By Definition~\ref{def-(f_H,g_H)}, $(f_H,g_H)(v_i)=(4m,2m)$ for $i \in \{1,6\}$, $(f_H,g_H)(v_j)=(5m,3m)$ for $j \in \{4,5\}$,  $(f_H,g_H)(v_2)=(12m,4m)$, $(f_H,g_H)(v_3)=(10m,4m)$.  
		Then
		\[\mathsf{Red}\bigl( 
		\langle v_4\mid v_1,m\rangle,\
		\langle \{v_1, v_3\} \mid v_2, m^*\rangle,\  
		\langle v_2\rangle,\ \langle v_5 \rangle,\ \langle v_6\rangle,\ 
		\langle v_1 \rangle,\ 
		\langle v_3\rangle,\ \langle v_4 \rangle
		\bigr)\]  is legal for $(H, L_H, g_H)$ and deletes all vertices of $H$,  which contradicts Lemma~\ref{key lemma}.  
	\end{proof} 
	
	\begin{claim}
		\label{FC-6-nonadjacent (6,3,3,4)}
		A $6$-vertex cannot be incident with two nonadjacent $(6,3,3,4)$-cycles. 
	\end{claim}
	\begin{proof}
		Assume to the contrary, $v_1v_2v_3v_4v_1$ and $v_1v_5v_6v_7v_1$ are two   $4$-cycles of $G$,
		where $d(v_i)=3$ for $i \in \{2,3,5,6\}$, $d(v_1)=6$, and $d(v_4)=d(v_7)=4$.  By Claim~\ref{FC-star}, there are no edges between $\{v_2,v_3\}$ and $\{v_5,v_6\}$. By Claim~\ref{FC-k_{k-2}-3_1}, $v_4v_6,v_7v_3 \notin E(G)$.
		Let $H=G[v_1,v_2, \ldots, v_7]$.
		By Definition~\ref{def-(f_H,g_H)},
		$(f_H,g_H)(v_1)=(5m,2m)$,  $(f_H,g_H)(v_i)=(9m,4m)$ for $i\in\{2,5\}$,   $(f_H,g_H)(v_j)=(10m,4m)$ for $j\in\{3,6\}$ and  $(f_H,g_H)(v_k)=(5m,3m)$ for $k\in\{4,7\}$. 
		Then
		\[\mathsf{Red}\bigl(\langle \{v_2, v_4\} \mid v_3, m^*\rangle,\ \langle v_3\rangle,\  
		\langle v_2\rangle,\  \langle v_4\rangle,\ \langle \{v_5, v_7\} \mid v_6, m^*\rangle,\ \langle v_6\rangle,\  
		\langle v_5\rangle,\  \langle v_7\rangle,\ \langle v_1\rangle \bigr)\]   is legal for $(H, L_H, g_H)$ and deletes all vertices of $H$,  which contradicts Lemma~\ref{key lemma}.  
	\end{proof}

	\section{Discharging}
	\label{sec-discharging}
	
	In this section, we prove Theorem~\ref{main-thm} using the discharging method.
	Let $G$ be a minimal counterexample to the theorem, embedded in the plane. By the minimality of $G$, it is connected. We define the initial charge of a vertex $v \in V(G)$ as  $c(v) = 3d(v) - 10$, and the initial charge of a face $f\in F(G)$ as $c(f) = 2d(f) - 10$. 
	By Euler's Formula, $$\sum_{z \in V(G) \cup F(G)} c(z) = \sum_{v \in V(G)} \big(3d(v) - 10\big) + \sum_{f \in F(G)} \big(2d(f) - 10\big) = 10(|E(G)|-|V(G)|-|F(G)|)=-20.$$ 
	
	Before introducing the discharging rules, we establish some necessary notation. 
	
	Let $uu_1u_2u_3u$ be the cycle bounding a $4$-face $f$.  We write $\Lambda_u(f) = (d(u), k_1, k_2, k_3)$, where $d(u_i) = k_i$ for $i=1, 2, 3$, and we say $f$ is a $(d(u), k_1, k_2, k_3)$-face. Here, each $k_i$ may be replaced by the notation defined at the beginning of Section~\ref{sec-RC} to specify a degree property. For instance, $\Lambda_u(f)=(5, 4, 5^+, 4_2)$ means that $uu_1u_2u_3u$ bounds a 4-face face $f$ with $d(u_1)=4$, $d(u_2) \ge 5$, and $u_3$ is a $4_2$-vertex.
	
	A $4$-face $f$ incident with exactly one $5^+$-vertex is said to be of
	\emph{type 1} if it is incident with exactly two $3$-vertices and one $4$-vertex;
	of \emph{type 2} if it is incident with two nonadjacent $4$-vertices and one $3_1$-vertex;
	and of \emph{type 3} if it is a $(5^+,3_0,4_1,4_1)$-face, $(5^+,3_0,4_2,4_0)$-face, a $(5^+,3_1,4_1,4_0)$-face,
	or a $(5^+,4_2,3_0,4_1)$-face.
	For $k \in \{2,3\}$, a face is of \emph{type $k^-$} if it is of type $i$ for some
	$i \le k$; otherwise, it is not of type $k^-$.

	For each vertex $u \in V(G)$, let $T_u(z)$ denote the amount of charge 
	transferred from $u$ to a neighbor or an incident face $z$, and for 
	$z_1,\ldots,z_k$, define
	$T_u(z_1,\ldots,z_k)=\sum_{i=1}^{k} T_u(z_i)$.
	
	Since $G$ is a plane graph, the neighbors of any vertex $v$ are naturally ordered by their cyclic clockwise position. Two neighbors of $v$ are said to be \emph{consecutive} if they appear consecutively in this cyclic order.

	\medskip
	\noindent 
	{\bfseries Discharging Rules.}  
	\begin{enumerate}[label={\bfseries{R\arabic*}},  ref=R\arabic*, leftmargin=*]
		\item \label{R1}
		Let $u$ be a $3$-vertex. 
		If $u$ has exactly one $3$-neighbor, then it receives $1/2$ charge 
		from each $4^+$-neighbor. 
		If $u$ has no $3$-neighbors, then it receives $1/3$ charge 
		from each neighbor.
		
		\item \label{R2}
		Let $u$ be a $4$-vertex and let $f$ be a $4$-face incident with $u$.
		\begin{enumerate}[label=(\arabic*), ref=\ref{R2}(\arabic*), leftmargin=*]
			
			\item \label{R2a}
			If $u$ has no $3$-neighbors, then $T_u(f)=\frac{1}{2}$.
			
			\item \label{R2b}
			If $u$ has exactly one $3$-neighbor that is a $3_0$-vertex, then 
			$T_u(f)=\frac{1}{2}$ if that neighbor is incident with $f$, 
			and $T_u(f)=\frac{1}{3}$ otherwise.
			
			\item \label{R2c}
			If $u$ has exactly one $3$-neighbor $v$, where $v$ is a $3_1$-vertex 
			with $3$-neighbor $w$, then 
			$T_u(f)=\frac{1}{2}$ if $f$ is incident with both $v$ and $w$, 
			and $T_u(f)=\frac{1}{3}$ otherwise.
			
			\item \label{R2d}
			If $u$ has two non-consecutive $3$-neighbors, then 
			$T_u(f)=\frac{1}{3}$.
			
			\item \label{R2e}
			If $u$ has two consecutive $3$-neighbors $v$ and $w$, then 
			$T_u(f)=\frac{1}{2}$ if $f$ is incident with both $v$ and $w$, 
			$T_u(f)=\frac{1}{3}$ if $f$ is incident with exactly one of $v$ and $w$, 
			and $T_u(f)=\frac{1}{6}$ otherwise.
		\end{enumerate}
		
		\item \label{R3} Let $u$ be a $5^+$-vertex and $f$ be a $4$-face incident with $u$. If $f$ is of type $k \in [3]$, then $T_u(f) = (10-k)/6$ (i.e., $3/2$, $4/3$, $7/6$ for $k=1, 2, 3$, respectively).
		
		\item \label{R4} Let $u$ be a $6^+$-vertex. If $f$ is a $4$-face incident with $u$ that is not of type $3^-$, then $T_u(f) = 1$.

		\item \label{R5}
		Let $u$ be a $5$-vertex and $f$ a $4$-face incident with $u$ that is not of type $3^-$. 
		Then $T_u(f)=t$ if $\Lambda_u(f)\in\mathcal{F}_t$, where 
		$t\in\{1,\,5/6,\,3/4,\,2/3,\,7/12,\,1/2\}$; the families $\mathcal{F}_t$ are defined as follows. 
		\begin{align*}
			\mathcal{F}_1 &= \{ (5,3,3,5^+),(5,3_0,4_1,4_0), (5,3,5^+,4_2), (5,3,5^+,3), (5,4_1,3_0,4_1), (5_0,4,3,5) \}; \\
			\mathcal{F}_{5/6} &= \{ (5,3_0,4_2,5_{\ge 1}), (5,3_1,4_1,5_{\ge 1}), (5,3,5^+,4_1),   (5_{\ge 1},4_2,3_0,5), (5_{\ge 1},4_1,3_1,5), \\
			&\phantom{=\{ \;}  (5,4_2,5, 4_2), (5,4_0,4_1,4_1), (5,4_1,4_0,4_1), (5,4_0,4_2,4_0), (5,4_2,4_0,5), (5,4_0,4_0,4_2) \}; \\
			\mathcal{F}_{3/4} &= \{ (5,3,5,4_0), (5,3,5,5),  (5,3_0,4_1,5_{\ge 1}), (5,4_1,5, 4_2), (5_{\geq 1},4_1,3_0,5) \}; \\
			\mathcal{F}_{2/3} &= \{ (5,3_1,4_1,5_0), (5,3_0,4_2,5_0), (5,3,4,6^+), (5,4,3,6^+), (5,4_0,4_1,4_0), (5,4_0,4_0,4_1),\\
			&\phantom{=\{ \;}  (5,4_1,4,5),  (5,4_1,5, 4_1), (5,4_2,5, 4_0), (5,4_2,5,5),(5,4_2,6^+, 4_2) \}; \\
			\mathcal{F}_{7/12} &= \{(5,4_1,5,4_0), (5,4_1,5,5)\};
		\end{align*}
		and let $\mathcal{F}_{1/2}$ consist of all remaining possible values of $\Lambda_u(f)$.
	\end{enumerate}

	\begin{figure}[h]
		\begin{minipage}[t]{0.19\textwidth}
			\centering
			\begin{tikzpicture}[>=latex,	
				whitenode/.style={circle, draw = black, minimum size=1.2mm, inner sep=1pt},
				rednode/.style={circle, draw = black, minimum size=1.2mm, inner sep=1pt}] 
				\node [rednode] (A1) at (0,1){};
				\node [rednode] (A2) at (1,2){};
				\node [rednode] (A3) at (2,1){}; 
				\node [rednode] (A4) at (1,0){}; 
				\node [whitenode] (B) at (1,1){\scriptsize $4$};
				\draw (A1)--(B)--(A2);
				\draw (A4)--(B)--(A3);
				
				\node at (2.4, 1){$4^+$};
				\node at (-0.3, 1){$4^+$};
				\node at (1.1, 2.4){$4^+$};
				\node at (1.1, -0.4){$4^+$};
				
				\draw[->] (B)--(0.6,1.4);
				\node at (0.4, 1.6){\re{$\frac{1}{2}$}};
				\draw[->] (B)--(1.4,1.4);
				\node at (1.6, 1.6){\re{$\frac{1}{2}$}};
				\draw[->] (B)--(0.6,0.6);
				\node at (0.4, 0.4){\re{$\frac{1}{2}$}};
				\draw[->] (B)--(1.4,0.6);
				\node at (1.6, 0.4){\re{$\frac{1}{2}$}};
			\end{tikzpicture} 
		\end{minipage}  
		\begin{minipage}[t]{0.19\textwidth}
			\centering
			\begin{tikzpicture}[>=latex,	
				whitenode/.style={circle, draw = black, minimum size=1.2mm, inner sep=1pt},
				graynode/.style={circle, draw = gray, fill=gray, minimum size=1.2mm, inner sep=1pt},
				rednode/.style={circle, draw = black,   minimum size=1.2mm, inner sep=1pt}] 
				\node [rednode] (A1) at (0,1){};
				\node [rednode] (A2) at (1,2){};
				\node [rednode] (A3) at (2,1){}; 
				\node [rednode] (A4) at (1,0){}; 
				\node [whitenode] (B) at (1,1){\scriptsize $4$};
				\draw (A1)--(B)--(A2);
				\draw (A4)--(B)--(A3);
				
				\node [whitenode] (C1) at (0,2){};
				\node [whitenode] (C2) at (2,2){};  
				\draw [line width=1pt] (C1)--(A2)--(C2);
				
				\node at (2.4, 1){$4^+$};
				\node at (-0.3, 1){$4^+$};
				\node at (1.1, 2.3){$3_0$};
				\node at (1.1, -0.4){$4^+$};
				
				\node at (-0.3, 2){$4^+$};
				\node at (2.4, 2){$4^+$};
				
				\draw[->] (B)--(0.6,1.4);
				\node at (0.4, 1.6){\re{$\frac{1}{2}$}};
				\draw[->] (B)--(1.4,1.4);
				\node at (1.6, 1.6){\re{$\frac{1}{2}$}};
				\draw[->] (B)--(0.6,0.6);
				\node at (0.4, 0.4){\bl{$\frac{1}{3}$}};
				\draw[->] (B)--(1.4,0.6);
				\node at (1.6, 0.4){\bl{$\frac{1}{3}$}};
			\end{tikzpicture} 
		\end{minipage} 
		\begin{minipage}[t]{0.19\textwidth}
			\centering
			\begin{tikzpicture}[>=latex,	
				whitenode/.style={circle, draw = black, minimum size=1.2mm, inner sep=1pt},
				rednode/.style={circle, draw = black, minimum size=1.2mm, inner sep=1pt}] 
				\node [rednode] (A1) at (0,1){};
				\node [rednode] (A2) at (1,2){};
				\node [rednode] (A3) at (2,1){}; 
				\node [rednode] (A4) at (1,0){}; 
				\node [whitenode] (B) at (1,1){\scriptsize $4$};
				\draw (A1)--(B)--(A2);
				\draw (A4)--(B)--(A3);
				
				\node at (2.4, 1){$4^+$};
				\node at (-0.3, 1){$4^+$};
				\node at (1.1, 2.4){$3$};
				\node at (1.1, -0.4){$3$};
				
				\draw[->] (B)--(0.6,1.4);
				\node at (0.4, 1.6){\bl{$\frac{1}{3}$}};
				\draw[->] (B)--(1.4,1.4);
				\node at (1.6, 1.6){\bl{$\frac{1}{3}$}};
				\draw[->] (B)--(0.6,0.6);
				\node at (0.4, 0.4){\bl{$\frac{1}{3}$}};
				\draw[->] (B)--(1.4,0.6);
				\node at (1.6, 0.4){\bl{$\frac{1}{3}$}};
			\end{tikzpicture} 
		\end{minipage}
		\begin{minipage}[t]{0.19\textwidth}
			\centering
			\begin{tikzpicture}[>=latex,	
				whitenode/.style={circle, draw = black, minimum size=1.2mm, inner sep=1pt},
				graynode/.style={circle, draw = gray, fill=gray, minimum size=1.2mm, inner sep=1pt},
				rednode/.style={circle, draw = black, minimum size=1.2mm, inner sep=1pt}] 
				\node [rednode] (A1) at (0,1){};
				\node [rednode] (A2) at (1,2){};
				\node [rednode] (A3) at (2,1){}; 
				\node [rednode] (A4) at (1,0){}; 
				\node [whitenode] (B) at (1,1){\scriptsize $4$};
				\draw (A1)--(B)--(A2);
				\draw (A4)--(B)--(A3);
				
				\node [whitenode] (C1) at (0,2){};
				\node [whitenode] (C2) at (2,2){};  
				\draw [line width=1pt] (C1)--(A2)--(C2);
				
				\node at (2.4, 1){$4^+$};
				\node at (-0.3, 1){$4^+$};
				\node at (1, 2.3){$3_1$};
				\node at (1.1, -0.4){$4^+$};
				
				\node at (-0.3, 2){$3$};
				\node at (2.4, 2){$4^+$};
				
				\draw[->] (B)--(0.6,1.4);
				\node at (0.4, 1.6){\re{$\frac{1}{2}$}};
				\draw[->] (B)--(1.4,1.4);
				\node at (1.6, 1.6){\bl{$\frac{1}{3}$}};
				\draw[->] (B)--(0.6,0.6);
				\node at (0.4, 0.4){\bl{$\frac{1}{3}$}};
				\draw[->] (B)--(1.4,0.6);
				\node at (1.6, 0.4){\bl{$\frac{1}{3}$}};
			\end{tikzpicture} 
		\end{minipage}
		\begin{minipage}[t]{0.19\textwidth}
			\centering
			\begin{tikzpicture}[>=latex,	
				whitenode/.style={circle, draw = black, minimum size=1.2mm, inner sep=1pt},
				rednode/.style={circle, draw = black, minimum size=1.2mm, inner sep=1pt}] 
				\node [rednode] (A1) at (0,1){};
				\node [rednode] (A2) at (1,2){};
				\node [rednode] (A3) at (2,1){}; 
				\node [rednode] (A4) at (1,0){}; 
				\node [whitenode] (B) at (1,1){\scriptsize $4$};
				\draw (A1)--(B)--(A2);
				\draw (A4)--(B)--(A3);
				
				\node at (2.4, 1){$4^+$};
				\node at (-0.3, 1){$3$};
				\node at (1.1, 2.4){$3$};
				\node at (1.1, -0.4){$4^+$};
				
				\draw[->] (B)--(0.6,1.4);
				\node at (0.4, 1.6){\re{$\frac{1}{2}$}};
				\draw[->] (B)--(1.4,1.4);
				\node at (1.6, 1.6){\bl{$\frac{1}{3}$}};
				\draw[->] (B)--(0.6,0.6);
				\node at (0.4, 0.4){\bl{$\frac{1}{3}$}};
				\draw[->] (B)--(1.4,0.6);
				\node at (1.6, 0.4){$\frac{1}{6}$};
			\end{tikzpicture} 
		\end{minipage}
		\caption{Illustration of Rule~\ref{R2}.}
	\end{figure}
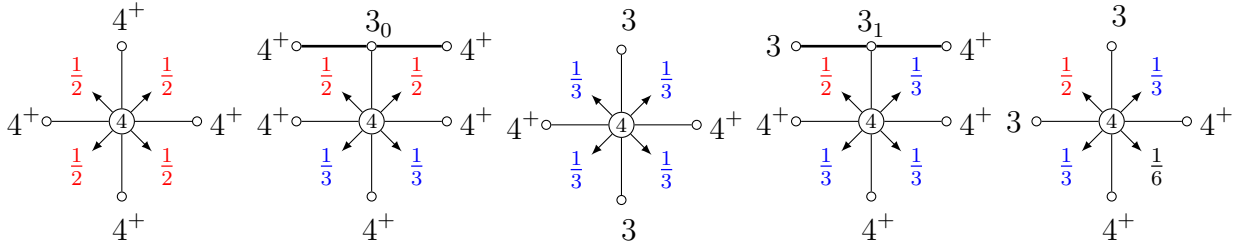

	For each $z \in V(G) \cup F(G)$, let $c^*(z)$ be the charge on $z$ after applying the discharging rules \ref{R1}--\ref{R5}, and in the next two sections, we show that $c^*(z) \geq 0$.

	\subsection{Final charges of the $4$-faces}
	
	Based on rules \ref{R1}--\ref{R5}, we make the following observations.
	
	\begin{observation} \label{transfer}
		Let $u$ be a vertex of $G$ and $f$ a $4$-face incident with $u$.
		\begin{enumerate}[label=(\arabic*)]
			\item \label{obs-1}
			If $d(u)\ge 4$, then $T_u(f)\ge \frac{1}{6}$.
			
			\item \label{obs-2}
			If $d(u)\ge 5$, then $T_u(f)\ge \frac{1}{2}$.
			
			\item \label{obs-3}
			If $d(u)\ge 6$, then $T_u(f)\ge 1$.
			
			\item \label{obs-4}
			If $d(u)\ge 5$ and there are two $3$-vertices on $f$, then $T_u(f)\ge 1$.
		\end{enumerate}
	\end{observation}

	\begin{lemma}
		\label{lem-4-face}
		If $f = v_1v_2v_3v_4v_1$ is a $4$-face, then $c^*(f) \geq 0$.
	\end{lemma}
	\begin{proof}
		First, note that $c(f)=2\times4-10=-2$. 
		If $d(v_i)\ge5$ for each $i\in[4]$, then by Observation~\ref{transfer}\ref{obs-2}, 
		$T_{v_i}(f)\ge\frac{1}{2}$, and hence 
		$c^*(f)\ge -2+\frac{1}{2}\times4=0$. 
		On the other hand, by Claim~\ref{FC-(4-,4-,4-,3-)}, 
		if $d(v_i)\le4$ for all $i\in[4]$, then necessarily $d(v_i)=4$ for each $i$. 
		Moreover, Claim~\ref{FC-(4,4,4,4_1/5_2)} implies that each $v_i$ has no $3$-neighbors. 
		Thus by Rule~\ref{R2a}, $T_{v_i}(f)=\frac{1}{2}$, and again 
		$c^*(f)\ge -2+\frac{1}{2}\times4=0$. 
		Hence we may assume, without loss of generality, that one of the following four cases occurs.

		\medskip
		\noindent
		{\bf Case 1.} $d(v_4) \leq 4$, and $d(v_i) \geq 5$ for each $i \in [3]$.
		\medskip
		
		By Observation \ref{transfer}\ref{obs-2}, we have $T_{v_i}(f) \geq \frac{1}{2}$ for $i\in [3]$ in this case. If there exists some $j\in [3]$ such that $d(v_j)\geq 6$, then $T_{v_j}(f) \geq 1$ by Observation \ref{transfer}\ref{obs-3}, and hence $c^*(f) \geq -2 + \frac{1}{2} \times 2 + 1 = 0$. Therefore, we may assume that
		$d(v_i)=5$ for each $i \in [3]$.
		
		If $d(v_4)=3$, then  for each $i \in \{1,3\}$, $\Lambda_{v_i}(f)=(5,3,5,5) \in \F_{3/4}$ and by Rule~\ref{R5},  $T_{v_i}(f) \geq \frac{3}{4}$. Thus $c^*(f) \geq -2 + \frac{3}{4} \times 2 + \frac{1}{2} = 0$.
		
		Now assume $d(v_4)=4$.
		If $v_4$ has no $3$-neighbors, then by Rule~\ref{R2a} and Observation \ref{transfer}\ref{obs-2}, $T_{v_i}(f) \geq \frac{1}{2}$ for each $i \in [4]$, thus $c^*(f) \geq - 2 + \frac{1}{2} \times 4=0$.
		If $v_4$ has exactly one $3$-neighbor, then by Rule~\ref{R2b} or \ref{R2c}  we have $T_{v_4}(f)= \frac{1}{3}$. Moreover, for $i \in \{1,3\}$, $\Lambda_{v_i}(f)=(5,4_1,5,5) \in \F_{7/12}$, so by Rule~\ref{R5}, $T_{v_i}(f)=\frac{7}{12}$. Hence $c^*(f) =-2 + \frac{1}{3} + \frac{7}{12} \times 2 + \frac{1}{2}=0$.
		If $v_4$ has two $3$-neighbors, then by Rule~\ref{R2e} we have  $T_{v_4}(f) =\frac{1}{6}$. Note that for $i \in \{1,3\}$, $\Lambda_{v_i}(f)=(5,4_2,5,5) \in \F_{2/3}$, so by Rule~\ref{R5}, $T_{v_i}(f) = \frac{2}{3}$. Thus $c^*(f) \geq -2 + \frac{1}{6} + \frac{2}{3} \times 2 + \frac{1}{2} =0$.

		\bigskip
		\noindent 
		{\bf Case 2.} $d(v_i) \leq 4$ for $i \in \{1,4\}$, and $d(v_j) \geq 5$ for $j \in \{2,3\}$.
		\medskip
		
		If $d(v_4) = d(v_1) =3$,  or both $d(v_2)\geq 6$ and $d(v_3)\geq 6$,  then by Observation \ref{transfer}\ref{obs-3} and \ref{obs-4}, $T_{v_j}(f) \geq 1$ for each $j \in \{2,3\}$, so $c^*(f) \geq -2 + 1 \times 2 = 0$.  Thus without loss of generality, assume that $d(v_4)=4$, and $d(v_2)=5$ or $d(v_3)=5$.
		
		\medskip
		\noindent
		{\bf Subcase 2.1.} $d(v_1)=3$.
		\medskip

		First, suppose that $v_1$ is a $3_1$-vertex or $v_4$ is a $4_2$-vertex (these two cannot occur simultaneously by Claim \ref{FC-k_{k-2}-3_1}). Then by Rule~\ref{R2c}, $T_{v_4}(f) = \frac{1}{3}$.  
		\begin{itemize}
			\item If $d(v_2)\geq 6$, then $d(v_3)=5$, and thus $\Lambda_{v_3}(f)=(5,4,3,6^+) \in \F_{2/3}$ and by Rule~\ref{R5}, $T_{v_3}(f) = \frac{2}{3}$. By Rule~\ref{R4}, $T_{v_2}(f)=  1$.  
			\item If $d(v_3)\geq 6$, then $d(v_2)=5$, and thus $\Lambda_{v_2}(f)=(5,3,4,6^+) \in \F_{2/3}$ and by Rule~\ref{R5}, $T_{v_2}(f) = \frac{2}{3}$. By Rule~\ref{R4}, $T_{v_3}(f)=  1$. 
		\end{itemize}  
		In either case,  $f$ receives $2$ charge in total, hence $c^*(f) \geq 0$. So assume $d(v_2)=d(v_3)=5$. 
		\begin{itemize}
			\item If $v_3$ is a $5_{\geq 1}$-vertex, then $\Lambda_{v_2}(f)\in \{(5,3_1,4_1,5_{\geq 1}),(5,3_0,4_2,5_{\geq 1})\}\subset \F_{5/6}$, and $\Lambda_{v_3}(f)\in\{(5_{\geq 1},4_2,3_0,5),(5_{\geq 1},4_1,3_1,5)\} \subset \F_{5/6}$.  By Rule~\ref{R5}, $T_{v_i}(f) = \frac{5}{6}$ for $i\in\{2,3\}$.  
			\item If $v_3$ is a $5_0$-vertex, then $\Lambda_{v_2}(f)\in \{(5,3_1,4_1,5_0),(5,3_0,4_2,5_0)\}\subset \F_{2/3}$, and $\Lambda_{v_3}(f)=(5_0,4,3,5) \in \F_{1}$. By Rule~\ref{R5}, $T_{v_2}(f) =\frac{2}{3}$ and $T_{v_3}(f) =1$. 
		\end{itemize} 
		In either case,  $f$ receives $2$ charge in total again, hence $c^*(f) \geq 0$.
		
		Next, suppose $v_1$ has no $3$-neighbors and $v_4$ has exactly one $3$-neighbor, namely $v_1$. By Rule~\ref{R2b}, $T_{v_4}(f) = \frac{1}{2}$. If $d(v_2)\geq 6$ or $d(v_3)\geq 6$, then by Observation \ref{transfer}\ref{obs-2} and \ref{obs-3}, $T_{v_2}(f) + T_{v_3}(f) \geq \frac{3}{2}$. Hence $c^*(f) \geq -2 +\frac{1}{2} + \frac{3}{2}= 0$.
		Thus we may assume $d(v_2)=d(v_3)=5$.  
		\begin{itemize}
			\item If $v_3$ is a $5_{\geq 1}$-vertex, then $\Lambda_{v_2}(f)=(5,3_0,4_1,5_{\geq 1})\in \F_{3/4}$, and $\Lambda_{v_3}(f)=(5_{\geq 1},4_1,3_0,5) \in \F_{3/4}$.
			By Rule~\ref{R5}, $T_{v_2}(f) = T_{v_3}(f) =\frac{3}{4}$. 
			\item Otherwise, $v_3$ is a $5_0$-vertex and $\Lambda_{v_3}(f)=(5_0,4,3,5) \in \F_{1}$. By Rule~\ref{R5}, $T_{v_2}(f) \geq \frac{1}{2}$ and $T_{v_3}(f) =1$. 
		\end{itemize} 
		In both situations,  $f$ receives at least $2$ in total charge, so $c^*(f) \geq 0$.

		\medskip
		\noindent
		{\bf Subcase 2.2.} $d(v_1)=4$.
		\medskip

		If there exists $j\in \{2,3\}$ such that $d(v_j)\ge 6$, then by
		Observation~\ref{transfer}\ref{obs-2} and~\ref{obs-3},
		we have $T_{v_2}(f) + T_{v_3}(f) \ge \frac{3}{2}$.
		To prove that $c^*(f)\ge 0$, it suffices to show that
		$T_{v_1}(f)+T_{v_4}(f) \ge \frac{1}{2}$. 
		This holds by Rule~\ref{R2a} if one of $v_1$ and $v_4$ is a $4_0$-vertex.
		Hence we may assume that both $v_1$ and $v_4$ have $3$-neighbors.
		By Claim~\ref{FC-5_3/4_24_1}, each of $v_1$ and $v_4$ has exactly one $3$-neighbor.
		Thus $T_{v_i}(f)\ge \frac{1}{3}$ for each $i\in\{1,4\}$ by
		Rule~\ref{R2b} or~\ref{R2c}. 
		
		Therefore, we may assume that $d(v_2)=d(v_3)=5$.
		
		If both $v_1$ and $v_4$ have no $3$-neighbors, then by Rule~\ref{R2a},
		$T_{v_i}(f)=\frac{1}{2}$ for each $i\in\{1,4\}$, and hence
		$c^*(f) \ge -2 + 4 \times \frac{1}{2} = 0$.
		Therefore,  we may assume without loss of generality that $v_1$ has at least one $3$-neighbor.
		
		If $v_1$ has two $3$-neighbors, then $v_4$ has no $3$-neighbors by Claim \ref{FC-5_3/4_24_1}. 
		Then $\Lambda_{v_2}(f)=(5,4_2,4_0,5)\in \F_{5/6}$.
		By Rule~\ref{R2a}, $T_{v_4}(f) = \frac{1}{2}$ and by Rule~\ref{R2e}, $T_{v_1}(f) = \frac{1}{6}$. By Rule~\ref{R5}, we have $T_{v_2}(f)=\frac{5}{6}$ and $T_{v_3}(f) \geq \frac{1}{2}$. Thus $c^*(f) \geq -2+\frac{1}{2}\times 2+\frac{1}{6} + \frac{5}{6}  =  0$.
		
		Now assume $v_1$ has exactly one $3$-neighbor. By Claim \ref{FC-5_3/4_24_1}, $v_4$  has at most one $3$-neighbor. If $v_4$  has a $3$-neighbor, then by Rule~\ref{R2b} or~\ref{R2c}, $T_{v_4}(f) =T_{v_1}(f) = \frac{1}{3}$.
		Since for each $i\in \{2,3\}$, $\Lambda_{v_i}(f)=(5,4_1,4,5)\in \F_{2/3}$, by Rule~\ref{R5}, we have $T_{v_2}(f) =T_{v_3}(f) = \frac{2}{3}$. Thus $c^*(f) \geq -2+\frac{1}{3}\times 2+\frac{2}{3} \times 2  =  0$.
		If $v_4$  has no $3$-neighbors, by Rule~\ref{R2a}, we have $T_{v_4}(f) =\frac{1}{2}$ and by Rule~\ref{R2b} or \ref{R2c}, $T_{v_1}(f) = \frac{1}{3}$.
		Since $\Lambda_{v_2}(f)=(5,4_1,4,5)\in \F_{2/3}$, by Rule~\ref{R5}, we have $T_{v_2}(f)=\frac{2}{3}$ and $T_{v_3}(f) \geq \frac{1}{2}$. Thus $c^*(f) \geq -2+\frac{1}{2}\times 2+\frac{1}{3} + \frac{2}{3}  =  0$.
		
		\bigskip
		\noindent 
		{\bf Case 3.} $d(v_i) \leq 4$ for $i \in \{2,4\}$, and $d(v_j) \geq 5$ for $j \in \{1,3\}$.
		\medskip
		
		If $d(v_2) = d(v_4) =3$,  or both $d(v_1)\geq 6$ and $d(v_3)\geq 6$,  then by Observation \ref{transfer}\ref{obs-3} and \ref{obs-4}, $T_{v_j}(f) \geq 1$ for each $j \in \{1,3\}$, so $c^*(f) \geq -2 + 1 \times 2 = 0$.  Thus without loss of generality, assume that $d(v_4)=4$, and $d(v_1)=5$ or $d(v_3)=5$.

		\medskip
		\noindent
		{\bf Subcase 3.1.} $d(v_2)= 3$.
		\medskip
		
		If $v_4$ has two $3$-neighbors, then  for each $i \in \{1,3\}$, either $d(v_i)\geq 6$ or $\Lambda_{v_i}(f)=(5,3,5^+,4_2)\in \F_{1}$, so by Rule~\ref{R4} or \ref{R5}, we have $T_{v_i}(f) \geq 1$, and hence $c^*(f) \geq -2 +1 \times 2 = 0$. If $v_4$ has exactly one $3$-neighbor, then  by Rule~\ref{R2b} or \ref{R2c}, $T_{v_4}(f) = \frac{1}{3}$. For each $i \in \{1,3\}$, either $d(v_i)\geq 6$ or $\Lambda_{v_i}(f)=(5,3,5^+,4_1)\in \F_{5/6}$, so by \ref{R4} or \ref{R5}, we have $T_{v_i}(f)\geq \frac{5}{6}$.
		Hence $c^*(f) \geq -2 + \frac{1}{3} + \frac{5}{6}\times 2  = 0$. If $v_4$ has no $3$-neighbors, then  by Rule~\ref{R2} $T_{v_4}(f) = \frac{1}{2}$.  If $d(v_1)\geq 6$ or $d(v_3)\geq 6$, then by
		Observation~\ref{transfer}\ref{obs-2} and~\ref{obs-3},
		$T_{v_1}(f) + T_{v_3}(f) \ge \frac{3}{2}$, so $c^*(f) \geq -2 + \frac{1}{2} \times 2 +1= 0$. Hence we may assume that  $d(v_1)=d(v_3)=5$. In this case, $\Lambda_{v_i}(f)=(5,3,5,4_0)\in \F_{3/4}$ for $i\in \{1,3\}$.  Therefore $c^*(f) \geq -2 + \frac{1}{2} + \frac{3}{4} \times 2= 0$.

		\medskip
		\noindent
		{\bf Subcase 3.2.} $d(v_2)= 4$.
		\medskip 
		
		Let $t_i$ denote the number of $3$-neighbors of $v_i$ for $i \in \{2,4\}$, and suppose that $t_4 \geq t_2$. If  $t_2 = t_4=0$, then by Rule~\ref{R2a}, $T_{v_i}(f) = \frac{1}{2}$ for each $i \in \{2,4\}$, so $c^*(f) \geq -2+\frac{1}{2} \times 4 =0$. If $t_2+t_4 \leq 3$ and $d(v_1)\geq 6$ or $d(v_3)\geq 6$, then by  Rule~\ref{R2}, we have $T_{v_2}(f) + T_{v_4}(f) \ge \frac{1}{3} + \frac{1}{6} = \frac{1}{2}$, and
		Observation~\ref{transfer}\ref{obs-2} and~\ref{obs-3},
		$T_{v_1}(f) + T_{v_3}(f) \ge \frac{3}{2}$, so  $c^*(f) \geq -2 +  \frac{1}{2} + \frac{3}{2}= 0$. Therefore, we assume that $t_4 \geq 1$, and once $t_2+t_4 \leq 3$ we have $d(v_1)=d(v_3)=5$.		 
		
		First assume that $t_4=1$ (hence $d(v_1)=d(v_3)=5$). By Rule~\ref{R2b} or \ref{R2c}, $T_{v_4}(f)=\frac{1}{3}$.
		\begin{itemize}
			\item If $t_2=1$, then for each $i \in \{1,3\}$, $\Lambda_{v_i}(f)=(5,4_1,5,4_1)\in \F_{2/3}$, and by Rule~\ref{R5}, $T_{v_1}(f)=T_{v_3}(f)=\frac{2}{3}$.
			By Rule~\ref{R2b} or \ref{R2c}, $T_{v_2}(f)=\frac{1}{3}$. Therefore $c^*(f) \geq -2+\frac{1}{3}\times 2 + \frac{2}{3}\times 2  =  0$. 
			\item If $t_2=0$, then by Rule~\ref{R2a}, $T_{v_2}(f)=\frac{1}{2}$. Moreover, for each $i \in \{1,3\}$,
			$\Lambda_{v_i}(f)=(5,4_1,5,4_0)\in \F_{7/12}$,
			so by Rule~\ref{R5},  $T_{v_i}(f)=\frac{7}{12}$. Thus $c^*(f) \geq -2+\frac{1}{3}+\frac{1}{2} + \frac{7}{12}\times 2  =  0$.
		\end{itemize}
		
		Then, assume that $t_2\leq 1$ and $t_4=2$, hence we have that $d(v_1)=d(v_3)=5$. 
		By Rule~\ref{R2e}, $T_{v_4}(f)=\frac{1}{6}$.  
		\begin{itemize}
			\item If $t_2=1$, then for each $i\in \{1,3\}$,   $\Lambda_{v_i}(f)=(5,4_1,5,4_2)\in \F_{3/4}$, so by Rule~\ref{R5},  $T_{v_1}(f)= T_{v_3}(f)=\frac{3}{4}$. By Rule~\ref{R2b} or \ref{R2c}, $T_{v_2}(f)=\frac{1}{3}$. Thus $c^*(f) \geq -2+\frac{1}{3}+\frac{1}{6}+ \frac{3}{4}\times 2 =  0$. 
			\item If $t_2=0$, then by Rule~\ref{R2a},  $T_{v_2}(f) = \frac{1}{2}$. For each $i\in \{1,3\}$,  $\Lambda_{v_i}(f)=(5,4_2,5,4_0)\in \F_{2/3}$, so by Rule~\ref{R5}, $T_{v_1}(f)= T_{v_3}(f)=\frac{2}{3}$.  Consequently, $c^*(f) \geq -2+\frac{1}{6}+\frac{1}{2} + \frac{2}{3}\times 2  =  0$.
		\end{itemize}
		
		Last we assume that $t_2=t_4=2$. By Rule~\ref{R2e}, $T_{v_2}(f) = T_{v_4}(f) = \frac{1}{6}$.
		\begin{itemize}
			\item If $d(v_1)=d(v_3)=5$, then for each $i\in \{1,3\}$, we have $\Lambda_{v_i}(f)=(5,4_2,5,4_2)\in \F_{5/6}$, so by Rule~\ref{R5}, $T_{v_1}(f)= T_{v_3}(f)=\frac{5}{6}$.  Thus $c^*(f) \geq -2+\frac{1}{6}\times 2+ \frac{5}{6}\times 2 =  0$.
			\item Otherwise, without loss of generality assume that $d(v_1)\geq 6$ and $d(v_3)=5$. Then $\Lambda_{v_3}(f)=(5,4_2,6^+,4_2)\in \F_{2/3}$. By Rule~\ref{R4}, $T_{v_1}(f)=1$ and by Rule~\ref{R5}, $T_{v_3}(f)=\frac{2}{3}$. Thus $c^*(f) \geq -2+\frac{1}{6}\times 2+ 1+\frac{2}{3} =  0$.
		\end{itemize}
		
		\medskip
		\noindent 
		{\bf Case 4.} $d(v_i) \leq 4$ for $i \in [3]$, and $d(v_4) \geq 5$.
		\medskip
		
		By Claim~\ref{FC-star},  it is impossible that $d(v_i) = 3$ for each $i \in [3]$.  If exactly two of $v_1, v_2, v_3$ have degree $3$, then  $f$ is of  type $1$. Thus by Rule~\ref{R3}, $T_{v_4}(f)=\frac{3}{2}$. Let $v_\ell$ be the unique $4$-vertex with $\ell \in [3]$, then by Rule~\ref{R2b} or \ref{R2c}, $T_{v_{\ell}}(f) = \frac{1}{2}$. Thus $c^*(f) \geq -2+\frac{1}{2} + \frac{3}{2} = 0$.
		So we consider the following two cases.
		
		\medskip
		\noindent 
		{\bf Subcase 4.1.}  Exactly one of $v_1, v_2, v_3$ has degree $3$.
		\medskip
		
		By the assumption, at least one of $d(v_1)=4$ and $d(v_3)=4$ hold, without loss of generality, assume that $d(v_3)=4$.  Also, either $d(v_1)=4$ and $d(v_2) = 3$, or $d(v_1) = 3$ and $d(v_2) = 4$. 
		
		First suppose that  $d(v_1)=4$ and $d(v_2)=3$. 
		\begin{itemize}
			\item If $v_2$ has a $3$-neighbor, then by Claim~\ref{FC-k_{k-2}-3_1}, neither $v_1$ nor $v_3$ has a $3$-neighbor. Hence by Rule~\ref{R2c}, $T_{v_1}(f) = T_{v_3}(f) = \frac{1}{3}$. Since $f$ is of type $2$, by Rule~\ref{R3}, $T_{v_4}(f) = \frac{4}{3}$. Thus $c^*(f) \geq -2 + \frac{1}{3} \times 2 + \frac{4}{3} = 0$. 
			
			\item Now assume $v_2$ has no $3$-neighbors. If neither $v_1$ nor $v_3$ has a $3$-neighbor distinct from $v_2$, then by Rule~\ref{R2b}, $T_{v_1}(f) = T_{v_3}(f) = \frac{1}{2}$. Note that either $d(v_4)\geq 6$ or $\Lambda_{v_4}(f)=(5,4_1,3_0,4_1)\in \F_{1}$,  in both cases, $T_{v_4}(f)\geq 1$ by Rule~\ref{R4} or \ref{R5}.
			Hence $c^*(f) \geq -2 + \frac{1}{2} \times 2 + 1 = 0$. Otherwise, at least one of $v_1$ and $v_3$ has a $3$-neighbor distinct from $v_2$. Nevertheless, by Claim~\ref{FC-3434~43}, at most one of $v_1$ and $v_3$ can have such a $3$-neighbor.
			Thus, by Rules~\ref{R2b} and \ref{R2c}, we have $T_{v_1}(f)+T_{v_3}(f) \geq \frac{1}{2} + \frac{1}{3}=\frac{5}{6}$. 
			Since $f$ is a $(d(v_4),4_2,3_0,4_1)$-face, it is a type 3 face incident with $v_4$, by Rule~\ref{R3}, $T_{v_4}(f) = \frac{7}{6}$. Consequently, $c^*(f) \geq -2 + \frac{5}{6} + \frac{7}{6}=0$.
		\end{itemize}
		
		Then assume that $d(v_1)=3$ and $d(v_2)=4$.  
		By Claim~\ref{FC-334~43}, at most one of $v_1$ and $v_3$ has a $3$-neighbor. 
		\begin{itemize}
			\item Assume that neither $v_1$ nor $v_3$ has a $3$-neighbor. By Rule~\ref{R2a}, $T_{v_3}(f)= \frac{1}{2}$.
			If $v_2$ has a $3$-neighbor distinct from $v_1$, then by Rule~\ref{R2d} or \ref{R2e}, $T_{v_2}(f)= \frac{1}{3}$. Since $f$ is a $(d(v_2),3_0,4_2,4_0)$-face, it is of type 3. Thus by Rule~\ref{R3}, $T_{v_4}(f) = \frac{7}{6}$. Hence $c^*(f) \geq -2 + \frac{1}{2} + \frac{1}{3}+ \frac{7}{6}=0$. If $v_2$ has no $3$-neighbors other than $v_1$, then by Rule~\ref{R2b}, $T_{v_2}(f)= \frac{1}{2}$. Note that either $d(v_4)\geq 6$ or $\Lambda_{v_4}(f)=(5,3_0,4_1,4_0)\in \F_{1}$, in both cases,  $T_{v_4}(f) = 1$ by Rules~\ref{R4} and \ref{R5}. Therefore $c^*(f) \geq -2 + \frac{1}{2}\times 2 + 1=0$.
			
			\item Assume that either $v_1$ or $v_3$ has a $3$-neighbor.
			If $v_3$ has a $3$-neighbor, then $v_2$ is a $4_1$-vertex by Claim~\ref{FC-5_3/4_24_1}, and $v_1$ is a $3_0$-vertex by Claim~\ref{FC-334~43}. So by Rule~\ref{R2b} or \ref{R2c}, $T_{v_3}(f)=\frac{1}{3}$, and by Rule~\ref{R2b}, $T_{v_2}(f) = \frac{1}{2}$. Since $f$ is a $(5,3_0,4_1,4_1)$-face, it is of type 3, and by Rule~\ref{R3}, $T_{v_4}(f) = \frac{7}{6}$. Hence $c^*(f) \geq -2 + \frac{1}{3} + \frac{1}{2} + \frac{7}{6}=0$.    
			If $v_1$ has a $3$-neighbor, then $v_3$ is a $4_0$-vertex. By Rule~\ref{R2a}, $T_{v_3}(f)=\frac{1}{2}$, and by Rule~\ref{R2c}, $T_{v_2}(f) =  \frac{1}{3}$. Since $f$ is a $(5,3_1,4_1,4_0)$-face, it is of type 3, and by Rule~\ref{R3}, $T_{v_4}(f) = \frac{7}{6}$. Hence $c^*(f) \geq -2 + \frac{1}{3} + \frac{1}{2} + \frac{7}{6}=0$. 
		\end{itemize}
		
		\medskip
		\noindent 
		{\bf Subcase 4.2.}  Each of $v_1, v_2, v_3$ has degree $4$.
		\medskip 
		
		If one of these three vertices has two $3$-neighbors, then by Claims~\ref{FC-5_3/4_24_1} and~\ref{FC-4_2-4/5_1-4_1}, the other two have no $3$-neighbors. Hence by Rule~\ref{R2},  $\sum_{i=1}^{3}T_{v_i}(f) \geq \frac{1}{6} + \frac{1}{2} + \frac{1}{2} = \frac{7}{6}$. Moreover, either $d(v_4)\geq 6$, or $\Lambda_{v_4}(f) \in \{(5,4_0,4_2,4_0),(5,4_0,4_0,4_2)\}\subset \F_{5/6}$, so  by Rule~\ref{R5}, $T_{v_4}(f) = \frac{5}{6}$. Therefore $c^*(f) \geq -2 +\frac{7}{6} + \frac{5}{6} = 0$. We may assume $v_i$ has at most one $3$-neighbor for each $i \in [3]$. Nevertheless, by Claim~\ref{FC-4_14_14_1}, at least one of $v_1$, $v_2$, and $v_3$ has no $3$-neighbors.  
		\begin{itemize}
			\item If all of $v_1$, $v_2$, and $v_3$ have no $3$-neighbor, then by Rules~\ref{R2a} and \ref{R5}, $T_{v_j}(f) \geq \frac{1}{2}$ for each  $j \in [4]$, hence $c^*(f) \geq -2+\frac{1}{2} \times 4 =0$. 
			
			\item If exactly two of $v_1$, $v_2$, and $v_3$ have no $3$-neighbor, then by Rules~\ref{R2a}--\ref{R2c}, we have 
			$\sum_{i=1}^{3}T_{v_i}(f)  \geq \frac{1}{3} + \frac{1}{2} + \frac{1}{2} = \frac{4}{3}$. Since $\Lambda_{v_4}(f) \in \{(5,4_0,4_0,4_1),(5,4_0,4_1,4_0)\}\subset \F_{2/3}$, by Rule~\ref{R5} we have $T_{v_4}(f) = \frac{2}{3}$.
			Thus $c^*(f) \geq -2 +\frac{4}{3} + \frac{2}{3} = 0$. 
			
			\item If exactly one of $v_1$, $v_2$, and $v_3$ has no $3$-neighbors, then  by Rules~\ref{R2a}--\ref{R2c}, we have
			$\sum_{i=1}^{3}T_{v_i}(f)  \geq \frac{1}{3} + \frac{1}{3} + \frac{1}{2} = \frac{7}{6}$. Since $\Lambda_{v_4}(f) \in \{(5,4_1,4_0,4_1),(5,4_1,4_1,4_0)\}\subset \F_{5/6}$, by Rule~\ref{R5} we have $T_{v_4}(f) \geq \frac{5}{6}$. Therefore, $c^*(f) \geq -2 +\frac{7}{6} + \frac{5}{6} = 0$.  
		\end{itemize}
		Therefore, we complete the proof of Lemma \ref{lem-4-face}.
	\end{proof}
	
	\subsection{Final charges of the vertices}
	
	\begin{lemma}
		\label{lem-3-vertex}
		If $u$ is $3$-vertex, then $c^*(u) \geq 0$.
	\end{lemma}
	\begin{proof}
		In this case, $c(u)=3\times 3-10=-1$.
		By Claim~\ref{FC-star}, $u$ has at most one $3$-neighbor. If $u$ has exactly one $3$-neighbor, then $u$ gets $\frac{1}{2}$ charge from each $4^+$-neighbor by Rule~\ref{R1}, hence $c^*(u) = -1 + \frac{1}{2} \times 2=0$. If $u$ has no $3$-neighbors, then $u$ gets $\frac{1}{3}$ charge from each neighbor by Rule~\ref{R1} again, hence $c^*(u) \geq -1+ \frac{1}{3} \times 3=0$. 
	\end{proof}
	
	\begin{lemma}
		\label{lem-4-vertex}
		If $u$ is a $4$-vertex, then $c^*(u) \geq 0$.
	\end{lemma}
	\begin{proof}
		The initial charge for a $4$-vertex $u$ is $c(u) = 3\times 4-10 =2$. By Claim~\ref{FC-star}, $u$ has at most two $3$-neighbors. 
		
		If $u$ has no $3$-neighbors, then by Rule~\ref{R2a}, $u$ sends $\frac{1}{2}$ charge to each incident face, hence $c^*(u)= 2-\frac{1}{2}\times 4=0$. 
		
		Assume that $u$ has exactly one $3$-neighbor $x$. If $x$ has no $3$-neighbors, then by Rule~\ref{R1}, $T_u(x) = \frac{1}{3}$, and by Rule~\ref{R2b}, $u$ sends $\frac{1}{2} \times 2 + \frac{1}{3} \times 2 = \frac{5}{3}$ charge to its incident faces.  So 
		$c^*(u) \geq 2- \frac{1}{3} - \frac{5}{3}= 0$. If $x$ has a $3$-neighbor $y$, then by Rule~\ref{R1}, $T_u(x) = \frac{1}{2}$ and by Rule~\ref{R2c}, $u$ sends $\frac{1}{2}  + \frac{1}{3} \times 3 = \frac{3}{2}$ charge to its incident faces, hence $c^*(u) \geq 2- \frac{1}{2} - \frac{3}{2}= 0$. 
		
		Then assume that $u$ has two $3$-neighbors. By Claim~\ref{FC-k_{k-2}-3_1}, each of these two $3$-vertices has no $3$-neighbors, thus $u$ sends $\frac{2}{3}$ charge to its neighbors in total. By Rule~\ref{R2d}, if the two $3$-neighbors are not consecutive around $u$, then $u$ sends $\frac{1}{3} \times 4=\frac{4}{3}$ charge to its incident faces. If the two $3$-neighbors are consecutive around $u$, then by Rule~\ref{R2e} $u$ sends $\frac{1}{2} + \frac{1}{3} \times 2 + \frac{1}{6}=\frac{4}{3}$ charge to its incident faces. Thus in either case, $c^*(u) = 0$.
	\end{proof}

	In the remainder of this paper, for a $k$-vertex $u$, let $v_0, v_1, \ldots, v_{k-1}$
	be the neighbors of $u$ listed in cyclic order. 
	For each $i \in \{0,1, \ldots, k-1\}$, let $f_i$ be the face incident with $u$
	that is bounded by the edges $uv_i$ and $uv_{i+1}$, where the indices are taken modulo $k$.
	Whenever $f_i$ is a $4$-face, denote its fourth vertex by $w_i$ so that $f_i = u v_i w_i v_{i+1} u$. Recall that $f_i$ is called a $(k,d(v_i),d(w_i),d(v_{i+1}))$-face. Note that $f_i$ and $f_{i+1}$ can only share exactly the edge $uv_{i+1}$; otherwise, $G$ would contain a vertex of degree $2$, contradicting Claim~\ref{FC-min-degree}.
	Let $r_0$ be the number of $(k,3,4,3)$-faces and $r_1$ be the number of  $(k,3,3,4)$-faces incident with $u$,
	and for each $i \in \{2,3\}$, let $r_i$ be the number of type $i$ faces incident with $u$. Furthermore, let $\alpha$ and $\beta$ be the numbers of $3_1$-neighbors and $3_0$-neighbors of $u$,
	respectively.
	
	\begin{observation}
		\label{obs-adjacent type 2-}
		Suppose $u$ is a $5^+$-vertex of $G$ and $f_i$ is of type $2^-$ for some $0 \leq i \leq d(u)-1$. Then neither $f_{i-1}$ nor $f_{i+1}$ is a type $2^-$ face. As a corollary, $r_0+r_1+r_2 \leq \lfloor d(u)/2  \rfloor$. 
	\end{observation}
	\begin{proof}
		Without loss of generality, suppose to the contrary that 
		$f_1$ and $f_2$ are two type $2^-$ faces incident with $u$ 
		that share the edge $uv_2$.  
		If one of $f_1$ and $f_2$ is of type~2, say $f_2$, 
		then $w_1v_2w_2$ is contained in a $(3,3,4,3)$-path, 
		contradicting Claim~\ref{FC-k_{k-2}-3_1}.   
		Thus assume that both $f_1$ and $f_2$ are of type~1. 
		Then, by the definition of type~1 faces, 
		either $v_1w_1v_2w_2v_3$ is a $(3,4,3,4,3)$-path, 
		contradicting Claim~\ref{FC-3434~43}, 
		or $w_1v_2w_2$ is a $(3,3,3)$-path, contradicting Claim~\ref{FC-star}, 
		or $w_1v_2w_2$ is contained in a $(3,3,4,3)$-path, 
		contradicting Claim~\ref{FC-k_{k-2}-3_1}. 
	\end{proof}

	\begin{observation}
		\label{obs-5343 or type 2}
		Suppose $u$ is a $5^+$-vertex of $G$ and $f_i$ is a $(d(u),3,4,3)$-face or a type $2$ face for some $0 \leq i \leq d(u)-1$. Then neither $f_{i-1}$ nor $f_{i+1}$ is a type $3^-$ face. 
	\end{observation}
	\begin{proof}
		By Observation~\ref{obs-adjacent type 2-}, neither $f_{i-1}$ nor $f_{i+1}$ 
		is a type $2^-$ face. 
		Without loss of generality, suppose to the contrary that $i=1$ and 
		$f_2$ is a type~3 face. 
		First consider the case that $f_1$ is of type~2. 
		Then $d(v_2)=4$ and $w_1$ is a $3_1$-vertex. 
		Either $d(w_2)=3$ and $d(v_3)=4$, in which case 
		$w_1v_2w_2$ is contained in a $(3,3,4,3)$-path, 
		contradicting Claim~\ref{FC-k_{k-2}-3_1}, 
		or $d(w_2)=4$ and $d(v_3)=3$, in which case 
		$v_3w_2v_2w_1$ is contained in a $(3,3,4,4,3)$-path, 
		contradicting Claim~\ref{FC-3434~43}. 
		Now assume that $f_1$ is a $(d(u),3,4,3)$-face; that is,
		$d(v_1)=d(v_2)=3$ and $d(w_1)=4$.
		Since $f_2$ is of type~3, we have $d(w_2)=d(v_3)=4$, and at least one of $w_2$ and $v_3$ has a $3$-neighbor. If $w_2$ has a $3$-neighbor, then $v_1w_1v_2w_2$ is contained in a $(3,4,3,4,3)$-path. If $v_3$ has a $3$-neighbor, then $v_1w_1v_2w_2v_3$ is contained in a $(3,4,3,4,4,3)$-path. Both possibilities contradict Claim~\ref{FC-3434~43}.
	\end{proof}

	\subsubsection{The $6^+$-vertices}
	
	In this section, let $u$ be a $6^+$-vertex, we will use the following observation.
	
	\begin{observation}
		\label{obs-number of $3$-neighbors} 
		$r_1+2r_2+r_3+ \alpha + \beta \leq d(u)$.
	\end{observation}
	\begin{proof}
		Since $u$ has exactly $\alpha+\beta$ neighbors of degree $3$, 
		it suffices to show that the number of $4$-neighbors of $u$ 
		is at least $r_1+2r_2+r_3$. 
		Each $(d(u),3,3,4)$-face, as well as each type~2 or type~3 face incident with $u$, 
		contains at least one $4$-neighbor of $u$; moreover, each type~2 face 
		contributes two $4$-neighbors of $u$. 
		By Observation~\ref{obs-adjacent type 2-}, if $f_i$ is of type $2^-$, 
		then neither $f_{i-1}$ nor $f_{i+1}$ is of type~$2^-$.  
		For two adjacent faces $f_i$ and $f_{i+1}$ that are either 
		$(d(u),3,3,4)$-faces or type~3 faces, 
		they together contribute at least two distinct $4$-neighbors of $u$. 
		Indeed, otherwise we would have $d(v_{i+1})=4$ and 
		$d(v_i)=d(v_{i+2})=3$, which implies 
		$d(w_i)=d(w_{i+1})=4$ by Claims~\ref{FC-k_{k-2}-3_1} 
		and~\ref{FC-334~43}.  
		Since $v_i$ has no $3$-neighbors by Claim~\ref{FC-334~43}, 
		it follows from the assumption on $f_i$ that either 
		$w_i$ is a $4_2$-vertex or $v_{i+1}$ is a $4_1$-vertex. 
		In either case, $w_iv_{i+1}w_{i+1}$ forms a $(4_2,4,4_1)$-path 
		or a $(4_1,4_1,4_1)$-path, contradicting 
		Claim~\ref{FC-4_2-4/5_1-4_1} or Claim~\ref{FC-4_14_14_1}.  
		Thus, the number of $4$-neighbors of $u$ is at least 
		$r_1+2r_2+r_3$. 
	\end{proof}
	
	By Observation~\ref{obs-number of $3$-neighbors}, 
	$\alpha+\beta \le d(u)-(r_1+2r_2+r_3)$. 
	By Claim~\ref{FC-k_{k-2}-3_1}, $G$ contains no $(3,3,4,3)$-path, 
	and thus $\beta \ge 2r_0$. 
	Combining these inequalities, we obtain 
	\begin{align}
		c^*(u) & \notag \geq 3d(u)-10-\Big[ \frac{3}{2}(r_0+r_1)+\frac{4}{3}r_2+ \frac{7}{6}r_3 
		+ (d(u)-r_0-r_1-r_2-r_3)+\frac{1}{2}\alpha + \frac{1}{3} \beta \Big] \\  
		& \notag = 2d(u)- 10 -\frac{1}{2}r_0 - \frac{1}{2}r_1-\frac{1}{3}r_2 - \frac{1}{6}r_3 
		- \frac{1}{2} (\alpha + \beta) + \frac{\beta}{6} \\ 
		& \notag \geq 2d(u)- 10 -\frac{1}{2}r_0 - \frac{1}{2}r_1-\frac{1}{3}r_2 - \frac{1}{6}r_3 
		- \frac{1}{2}\big(d(u)- (r_1+2r_2+r_3)\big) + \frac{2r_0}{6} \\ 
		& = \frac{3}{2}d(u) -10 -  \frac{1}{6}r_0  + \frac{2}{3}r_2 + \frac{1}{3}r_3.
		\label{eq-7+}
	\end{align}

	\begin{lemma} \label{lem-7-vertex}
		If $d(u) \geq 7$, then $c^*(u) \geq 0$.
	\end{lemma}
	\begin{proof} 
		By Observation~\ref{obs-adjacent type 2-}, $r_0 \leq \lfloor \frac{d(u)}{2} \rfloor$. Substituting this into~\eqref{eq-7+} yields $c^*(u) \geq \frac{3}{2}d(u) -10 -  \frac{1}{6} \lfloor \frac{d(u)}{2} \rfloor \geq 0$ for $d(u) \geq 7$. 
	\end{proof}
	
	\begin{lemma} \label{lem-6-vertex}
		If $d(u)= 6$, then $c^*(u) \geq 0$.
	\end{lemma}
	\begin{proof} 
		In this case, we have $c(u)=3d(u) -10=8$.
		
		\medskip
		\noindent
		{\bf Case 1.} $G$ contains a $(6,3,4,3)$-face.
		\medskip   
		
		Without loss of generality, assume that $f_0$ is a $(6,3,4,3)$-face. Then $d(v_0)=d(v_1)=3$ and $d(w_0)=4$. By Claim~\ref{FC-k_{k-2}-3_1}, neither $v_0$ nor $v_1$ has a $3$-neighbor. Thus $T_u(f_0)=\frac{3}{2}$ by Rule~\ref{R3} and $T_u(v_0)=T_u(v_1)=\frac{1}{3}$ by Rule~\ref{R1}. By Observation~\ref{obs-5343 or type 2}, neither $f_1$ nor $f_5$ is a face of type $3^-$, so by Rule~\ref{R4} $T_u(f_i)\le 1$ for $i\in\{1,5\}$. Thus $T_u(f_0,f_1,f_5,v_0,v_1)\le \frac{3}{2}+1+1+\frac{1}{3}+\frac{1}{3}=\frac{25}{6}$. By Claim~\ref{FC-(5_3/6_4,3,4,3)}, $u$ has at most one additional $3$-neighbor besides $v_0$ and $v_1$. Hence for each $j\in\{2,3,4\}$, $f_j$ is not a $(6,3,4,3)$-face. Moreover, by Claim~\ref{FC-(5_2/6_3,3,3,4)}, $f_j$ is also not a $(6,3,3,4)$-face. Thus by Rule~\ref{R3} $T_u(f_j)\le \frac{4}{3}$. By Observation~\ref{obs-adjacent type 2-}, at most two of the faces $f_2,f_3,f_4$ are type $2$ faces. If $u$ has no additional $3$-neighbor, then $c^*(u)\ge 8-\frac{25}{6}-(\frac{4}{3}\times2+\frac{7}{6})=0$. Thus we may assume that $u$ has an additional $3$-neighbor $v_k$. By Rule~\ref{R1}, $T_u(v_k)\le \frac{1}{2}$. By symmetry, suppose $k\in\{2,3\}$. If $k=2$, then for $i\in\{3,4\}$, $f_i$ is not a $(6,4,3,4)$-face by Claim~\ref{FC-6-(6_3,4,3,4)}, and thus $T_u(f_i)\le 1$ by Rule~\ref{R4}. Consequently, $c^*(u)\ge 8-\frac{25}{6}-(1\times2+\frac{4}{3}+\frac{1}{2})=0$. If $k=3$, symmetrically $T_u(f_4)\le 1$ by Rule~\ref{R4}. Note that neither $f_2$ nor $f_3$ is a $(6,4,3,4)$-face, so by Rules~\ref{R3} and~\ref{R4} $T_u(f_2,f_3)\le \frac{7}{6}\times2=\frac{7}{3}$. Therefore, $c^*(u)\ge 8-\frac{25}{6}-(1+\frac{7}{3}+\frac{1}{2})=0$. 
		
		\medskip
		\noindent
		{\bf Case 2.} $G$ contains no $(6,3,4,3)$-face.
		\medskip 
		
		In this case, we have $r_0=0$. By Claim \ref{FC-6-nonadjacent (6,3,3,4)} and Observation \ref{obs-adjacent type 2-}, $r_1\leq 1$. Then $$c^*(u) \geq 8-\Big(\frac{3}{2}r_1+\frac{4}{3}r_2+ \frac{7}{6}r_3 
		+ (6-r_1-r_2-r_3)+\frac{1}{2}\alpha + \frac{1}{3}\beta\Big)=2-\big(\frac{1}{2}r_1+\frac{1}{3}r_2+\frac{1}{6}r_3\big)-\big(\frac{1}{2}\alpha +\frac{1}{3}\beta\big).$$ We may assume that $(r_2,r_3) \in \{(1,0), (0,0),(0,1), (0,2)\}$; otherwise Inequality~\eqref{eq-7+} already yields $c^*(u) \geq 0$. For these possibilities,  $\frac{1}{3}r_2+\frac{1}{6}r_3\leq \frac{1}{3}$.    
		If $r_1=1$, we have $\alpha+\beta\leq 2$ by Claim \ref{FC-(5_2/6_3,3,3,4)}. Then $\frac{1}{2}\alpha +\frac{1}{3}\beta\leq 1$. Hence  $c^*(u)\geq 2-(\frac{1}{2}+\frac{1}{3})-1=\frac{1}{6}>0$. If $r_1=0$, by Claim \ref{FC-star}, $\alpha+\beta\leq 4$. If $\alpha+\beta=4$, then Claim \ref{FC-k_{k-2}-3_1} implies $\alpha=0$, so $\frac{1}{2}\alpha +\frac{1}{3}\beta=\frac{1}{3}\beta=\frac{4}{3}$. If $\alpha+\beta\leq 3$, then $\frac{1}{2}\alpha +\frac{1}{3}\beta\leq \frac{3}{2}$. 
		Therefore the maximum possible value of $\frac{1}{2}\alpha +\frac{1}{3}\beta$ is $\frac{3}{2}$.  
		Consequently, $c^*(u)\geq 2-\frac{1}{3}-\frac{3}{2}=\frac{1}{6}>0$.
	\end{proof}	
	
	\subsubsection{The vertices of degree $5$}
	
	In this section, we focus on $5$-vertices. Let $u$ be a $5$-vertex, so $c(u)=3d(u) -10=5$.
	
	\begin{observation}
		\label{obs- almost 1/2}
		Suppose that $f_i=uv_iw_iv_{i+1}u$ is a $4$-face incident with $u$, where $d(v_i)\ge 4$ and $d(v_{i+1})\ge 4$, and, if $d(v_i)=4$ (resp.\ $d(v_{i+1})=4$), then $v_i$ (resp.\ $v_{i+1}$) has no $3$-neighbors (hence $d(w_i)\ge 4$). 
		If $f_i$ is a $(5,4_0,4_2,4_0)$-face, then $T_u(f_i)=\frac{5}{6}$; 
		if $f_i$ is a $(5,4_0,4_1,4_0)$-face, then $T_u(f_i)=\frac{2}{3}$; 
		otherwise, $T_u(f_i)\le \frac{1}{2}$. 
	\end{observation} 
	
	\begin{observation}
		\label{obs- almost 3/4}
		Suppose $f_i=uv_iw_iv_{i+1}u$ is a $4$-face incident with $u$, where $d(v_i)=3$, $d(w_i) \geq 4$, $d(v_{i+1}) \geq 4$, and if $d(v_{i+1}) = 4$, then $v_{i+1}$ has no $3$-neighbors. If $u$ has another $3$-neighbor $x$ distinct from $v_i$, then $T_u(f_i) \leq \frac{3}{4}$.
	\end{observation} 
	
	\begin{proof}
		By Claim \ref{FC-(5_2,3,4,4)}, $(d(w_i),d(v_{i+1}))\neq (4,4)$, hence $f_i$ is not of type $3^-$. First we assume $d(v_{i+1})=4$. Then $v_{i+1}$ has no $3$-neighbors and $d(w_i)\geq 5$, and hence $\Lambda_u(f_i)\notin \F_{1}\cup \F_{5/6}$, it follows that $T_u(f_i) \leq \frac{3}{4}$. Then we assume $d(v_{i+1})=5$. If $d(w_i)=4$, then by Claim \ref{FC-(5_2,3_1,4,5_1)}, $v_{i+1}$ has no $3$-neighbors, so $\Lambda_u(f_i) \neq (5,3_1,4_1,5_{\ge 1})$ and $\Lambda_u(f_i)\notin \F_{1}\cup \F_{5/6}$, and again $T_u(f_i) \leq \frac{3}{4}$. If $d(w_i)\geq 5$,  then $\Lambda_u(f_i)\notin \F_{1}\cup \F_{5/6}$, which implies $T_u(f_i) \leq \frac{3}{4}$. Now we assume $d(v_{i+1})\geq 6$, then $\Lambda_u(f_i)\notin \F_{1}\cup \F_{5/6}\cup \F_{3/4}$, so $T_u(f_i) \leq \frac{2}{3}$.
		In all cases, we obtain $T_u(f_i) \leq \frac{3}{4}$.
	\end{proof}  
	
	\begin{lemma}
		\label{lem-exist (5,3,4,3)-face}
		If $u$ is incident with a $(5,3,4,3)$-face, then $c^*(u) \geq 0$.
	\end{lemma}
	\begin{proof}
		Without loss of generality, assume that $f_0$ is a $(5,3,4,3)$-face; that is,
		$d(v_0)=d(v_1)=3$ and $d(w_0)=4$. 
		By Claim~\ref{FC-k_{k-2}-3_1}, neither $v_0$ nor $v_1$ has a $3$-neighbor.
		Hence $T_u(f_0,v_0,v_1)=\frac{3}{2}+\frac{1}{3}+\frac{1}{3}=\frac{13}{6}$.    
		By Observation~\ref{obs-5343 or type 2}, neither $f_1$ nor $f_4$ is of type $3^-$. 
		Moreover, Claim~\ref{FC-(5_2,3,4,4)} implies that $\Lambda_u(f_1)\neq(5,3,4,4)$, and hence $\Lambda_u(f_1)\notin\mathcal F_1$. 
		Thus $T_u(f_1)\le\frac{5}{6}$. By Claim~\ref{FC-(5_3/6_4,3,4,3)}, we have $d(v_i)\ge4$ for each $2\le i\le4$; 
		moreover, if $d(v_i)=4$, then $v_i$ has no $3$-neighbors by Claim~\ref{FC-(5,3,4,3)-4_1}. 
		By Claim~\ref{FC-(5_2,4,4_1,4)}, $\Lambda_u(f_2)\notin\{(5,4_0,4_1,4_0),(5,4_0,4_2,4_0)\}$;
		together with Observation~\ref{obs- almost 1/2}, this yields $T_u(f_2)\le\frac{1}{2}$. 
		By symmetry, $T_u(f_3)+T_u(f_4)\le \frac{1}{2}+\frac{5}{6}=\frac{4}{3}$. 
		Therefore, $c^*(u)\ge 5-\frac{13}{6}-2\times \frac{4}{3}=\frac{1}{6}>0$.  
	\end{proof}

	\begin{lemma}
		\label{lem-exist (5,3,3,4)-face}
		If $u$ is incident with a $(5,3,3,4)$-face, then $c^*(u) \geq 0$.
	\end{lemma}
	\begin{proof}
		Without loss of generality, assume that $f_0$ is a $(5,3,3,4)$-face, where $d(v_0)=d(w_0)=3$ and $d(v_1)=4$. Then $T_u(f_0,v_0,v_1) = \frac{3}{2} + \frac{1}{2} = 2$. 
		By Observation~\ref{obs-adjacent type 2-} and Claim~\ref{FC-adjcent-(5,3,3,4)&(5,3,4,4)}, $f_4$ is not a face of type $3^-$. Hence $T_u(f_4) \leq  1$.      
		By Claim~\ref{FC-(5_2/6_3,3,3,4)}, for each $2\leq i\leq 4$, we have $d(v_i)\geq 4$; moreover, if $d(v_i)=4$, then $v_i$ has no $3$-neighbors by Claim~\ref{FC-(5,3,3,4)-4_1}. 
		
		First we show $T_u(f_1) \leq  \frac{1}{2}$. We may assume $f_1$ is a 4-face; otherwise $T_u(f_1)=0$.
		By Claim \ref{FC-k_{k-2}-3_1}, $d(w_1)\geq 4$. Moreover, if $d(w_1)=4$, then $w_1$ also has no $3$-neighbors by Claim \ref{FC-334~43}. Therefore $T_u(f_1) \leq  \frac{1}{2}$ by Observation \ref{obs- almost 1/2}.
		Next we show that $T_u(f_2,f_3) \leq  \frac{3}{2}$. Suppose, for a contradiction, that by Observation \ref{obs- almost 1/2}, we have $T_u(f_2) = T_u(f_3) =\frac{5}{6}$. It then follows that $d(v_2)=d(v_3)=d(v_4)=4$, both $w_2$ and $w_3$ are $4_2$-vertices, contradicting Claim \ref{FC-4_2-4/5_1-4_1}. 
		
		Therefore,  $c^*(u) \geq 5- 2-1-\frac{1}{2}-\frac{3}{2} =0$.  
	\end{proof}   
	
	\begin{lemma}
		\label{lem-exist (5,4,3,4)-face}
		If $u$ is incident with a $(5,4,3,4)$-face, then $c^*(u) \geq 0$.
	\end{lemma}

	\begin{proof} 
		By Claim~\ref{FC-(5_1,4,3,4)}, we know that $d(v_i) \geq 4$ for each $i \in \{0,1,2,3,4\}$.
		
		If there are two $(5,4,3_1,4)$-faces incident with $u$, then they are not adjacent
		by Observation~\ref{obs-adjacent type 2-}.
		Assume these two faces are $f_0$ and $f_2$.
		We first show that $T_u(f_1) \leq \frac{2}{3}$. We may assume $f_1$ is a $4$-face; otherwise, $T_u(f_1)=0$. By Claim \ref{FC-k_{k-2}-3_1}, $d(w_1)\geq 4$, and moreover, by Claim \ref{FC-334~43}, we actually have $d(w_1)\geq 5$. If  $d(w_1)\geq 6$, then $\Lambda_u(f_1) = (5,4_1,6^+,4_1)\in \F_{1/2}$. If  $d(w_1)=5$, $\Lambda_u(f_1) = (5,4_1,5,4_1)\in \F_{2/3}$. In both cases, by Rule~\ref{R5}, we have  
		$T_u(f_1) \leq \frac{2}{3}$.
		By Claim \ref{FC-k_{k-2}-3_1}, $v_3$ has no $3$-neighbors distinct from $w_2$. Then $\Lambda_u(f_3) \notin \F_1$ and hence $T_u(f_3) \leq \frac{5}{6}$.          
		By symmetry, $T_u(f_4) \leq \frac{5}{6}$. Consequently,
		$c^*(u) \geq 5 - \frac{4}{3} \times 2 - \frac{2}{3} - \frac{5}{6} \times 2 = 0$.
		
		If exactly one $(5,4,3_1,4)$-face is incident with $u$, say $f_0$, 
		then, as in the previous paragraph, 
		$T_u(f_i)\le \frac{5}{6}$ for $i\in\{1,4\}$. 
		By Claim~\ref{FC-(5,4,3,4)-4_2}, neither $f_2$ nor $f_3$ is a 
		$(5,4_2,3_0,4)$-face. 
		Hence $T_u(f_j)\le 1$ for $j\in\{2,3\}$, and therefore
		$c^*(u)\ge 5-\frac{4}{3}-2\cdot\frac{5}{6}-2=0$.
		
		Hence we may assume that $u$ is not incident with a $(5,4,3_1,4)$-face. 
		If $u$ is not incident with a $(5,4_2,3_0,4)$-face, then 
		$T_u(f_i)\le 1$ for each $i\in\{0,1,2,3,4\}$, and hence 
		$c^*(u)\ge 5-5=0$. 
		Thus, without loss of generality, assume that $f_0$ is a 
		$(5,4_2,3_0,4)$-face such that $v_0$ has a $3$-neighbor distinct from $w_0$ 
		and $v_1$ has no other $3$-neighbors. 
		It follows that $\Lambda_u(f_1) \notin \F_1$ and hence $T_u(f_1) \leq \frac{5}{6}$.
		On the other hand, by Claim~\ref{FC-(5,4,3,4)-4_2}, neither $f_2$ nor $f_3$
		is a $(5,4_2,3_0,4)$-face.
		Thus $T_u(f_j) \leq 1$ for $j \in \{2,3\}$. Now consider $f_4$.
		If $f_4$ is not a $(5,4_2,3_0,4)$-face, then $T_u(f_4) \leq 1$ by Rule~\ref{R5}, and hence $c^*(u) \geq 5 - \frac{7}{6} - \frac{5}{6} - 1 \times 3 = 0$. If  $f_4$ is  a $(5,4_2,3_0,4)$-face, then $T_u(f_4) = \frac{7}{6}$, and symmetrically $T_u(f_3) \leq \frac{5}{6}$. Hence,
		$c^*(u) \geq 5 - \frac{7}{6} \times 2 - \frac{5}{6} \times 2 - 1 = 0$.
	\end{proof} 
	
	\begin{lemma}
		\label{lem-exist (5,3,4,4)-face}
		If $u$ is incident with a $(5,3,4,4)$-face, then $c^*(u) \geq 0$.
	\end{lemma} 
	\begin{proof}
		Without loss of generality, we assume that $f_0$ is a $(5,3,4,4)$-face, where $d(v_0)=3$ and $d(w_0)=d(v_1)=4$. For each $i \in \{2,3,4\}$, $d(v_i) \geq 4$ by Claim~\ref{FC-(5_2,3,4,4)}. By Lemmas~\ref{lem-exist (5,3,4,3)-face},~\ref{lem-exist (5,3,3,4)-face} and~\ref{lem-exist (5,4,3,4)-face}, we may assume that there is no type $2^-$ faces incident with $u$. Thus we observe that $T_u(v_0, f_4) \leq \frac{3}{2}$. Indeed, if $v_0$ is a $3_1$-vertex, then either $\Lambda_u(f_4)=(5,3,3,5^+)$ or $d(f_4)\geq 5$, hence $T_u(v_0, f_4) \leq  \frac{1}{2} + 1 =\frac{3}{2}$. If $v_0$ is a $3_0$-vertex, as $T_u(f_4) \leq \frac{7}{6}$, $T_u(v_0, f_4) \leq  \frac{1}{3} + \frac{7}{6} =\frac{3}{2}$ too.
		
		By Lemma~\ref{lem-exist (5,4,3,4)-face}, we may assume that $f_i$ is not a $(5,4,3,4)$-face for $i \in [3]$. Since $d(v_j) \geq 4$ for each $j \in [4]$, we have that $f_i$ is not of type $3^-$  for $i \in [3]$; moreover, $\Lambda_u(f_i) \notin \F_1$ for each $i \in [3]$ as $u$ is a $5_1$-vertex. By Claim~\ref{FC-4_2-4/5_1-4_1}, $\Lambda_u(f_i) \neq (5,4_{\geq 1},5,4_2)$. Thus we have for each $i \in [3]$, $$\Lambda_u(f_i)\in \{(5,4_0,4_1,4_1), (5,4_1,4_0,4_1),(5,4,3,5),(5,4_0,4_2,4_0),(5,4_2,4_0,5), (5,4_0,4_0,4_2)\}.$$ Indeed, for otherwise, there exists some $\ell \in [3]$ such that $\Lambda_u(f_\ell) \notin \F_1 \cup \F_{5/6} \cup \F_{3/4}$, hence $T_u(f_\ell) \leq \frac{2}{3}$ by Rule~\ref{R5}. This would imply that $c^*(u) \geq 5- \frac{7}{6}-\frac{3}{2} - \frac{5}{6} \times 2 - \frac{2}{3} = 0$. 
		
		We may assume that $T_u(v_0, f_4) > \frac{4}{3}$, for otherwise $c^*(u) \geq 5- \frac{7}{6}-\frac{4}{3} - \frac{5}{6} \times 3 = 0$. Thus $\Lambda_u(f_4) \in \{(5,3,3,5^+), (5,3,4,4_1)\}$. 
		
		If $v_0$ has a $3$-neighbor, then $\Lambda_u(f_4) = (5,3,3,5^+)$, where $d(v_4) \geq 5$. Moreover, by Claim~\ref{FC-334~43}, $v_1$ must be a $4_0$-vertex, $d(w_1)\geq 4$ and $w_1$ is not a $4_{\geq 1}$-vertex, thus $\Lambda_u(f_1) = (5,4_0,4_0,4_2)$, that is $v_2$ is a $4_2$-vertex. Also we have that $\Lambda_u(f_3) \in \{(5,4,3,5), (5,4_2,4_0,5)\}$, which implies that $v_3$ is a $4_{\geq 1}$-vertex. But then $v_2uv_3$ is a $(4_2, 4, 4_{\geq 1})$-path, which contradicts Claim~\ref{FC-4_2-4/5_1-4_1}.
		
		If $v_0$ has no $3$-neighbors, then $\Lambda_u(f_4) = (5,3,4,4_1)$. If $T_u(f_0) \leq 1$, then $c^*(u) \geq 5- 1-\frac{3}{2} - \frac{5}{6} \times 3 = 0$. Thus we have that $T_u(f_0)= \frac{ 7}{6}$, which implies that $v_1$ is a $4_1$-vertex. But then $v_4w_4v_0w_0v_1$ is contained in a $(3,4,4,3,4,4,3)$-path, contradicting Claim~\ref{FC-3443443}.
	\end{proof}
	
	\begin{lemma}
		\label{lem-5-vertex}
		If $d(u)=5$, then $c^*(u) \geq 0$.
	\end{lemma} 
	\begin{proof}  
		By Lemmas \ref{lem-exist (5,3,4,3)-face}--\ref{lem-exist (5,3,4,4)-face}, we assume that there is no face of type $3^-$ incident with $u$. Thus $T_u(f_i) \leq 1$ for each $i \in \{0,1,2,3,4\}$. 
		If $u$ has no $3$-neighbors, then $c^*(u) \geq 5- 1 \times 5= 0$.
		If $u$ has exactly one $3$-neighbor, without loss of generality, say $v_0$. Then for each $i \in [3]$, $\Lambda_u(f_i) \notin \F_1$, and hence $T_u(f_i) \leq \frac{5}{6}$. Therefore, $c^*(u) \geq 5 -\frac{1}{2}-1 \times 2- \frac{5}{6} \times 3  =  0$. 
		So assume that $u$ has at least two $3$-neighbors. By Claim \ref{FC-star}, $u$ has at most three $3$-neighbors.
		
		\medskip
		\noindent
		{\bf Case 1.} $u$ has two $3$-neighbors.
		\medskip

		Without loss of generality, assume that $u$ has exactly two $3$-neighbors
		$v_0$ and $v_{\ell}$ for some $\ell \in \{1,2\}$.
		So for each $i \in [4] \setminus \{\ell\}$, $d(v_i) \geq 4$ 
		and $v_i$ is not a $4_2$-vertex by Claim~\ref{FC-5_24_2}.
		
		If neither $v_0$ nor $v_{\ell}$ has a $3$-neighbor, then
		$T_u(v_0,v_\ell) = \frac{2}{3}$.
		Moreover, for each $i \in [4]$, $\Lambda_u(f_i) \notin \F_1$, so $T_u(f_i) \leq \frac{5}{6}$.
		Thus $c^*(u) \geq 5- \frac{2}{3} - 1 - \frac{5}{6} \times 4 = 0$. 
		So without loss of generality, $v_0$ has a $3$-neighbor.
		By Claim~\ref{FC-3_15_24_1}, for each $j \in [4] \setminus \{\ell\}$,
		$v_j$ is not a $4_{\geq 1}$-vertex.
		
		If $\ell=1$, then by Observation~\ref{obs- almost 1/2} and 
		Claim~\ref{FC-(5_2,4,4_1,4)}, $T_u(f_2)\leq\frac{1}{2}$ and $T_u(f_3)\leq\frac{1}{2}$. Note that $T_u(v_0,v_1)\leq \frac{1}{2}\times 2=1$. Thus $c^*(u) \geq 5-1-\frac{1}{2}\times 2-1 \times 3=0$.

		Thus assume that $\ell=2$.
		By Observation~\ref{obs- almost 1/2} and Claim~\ref{FC-(5_2,4,4_1,4)}, $T_u(f_3)\leq\frac{1}{2}$. 	Note that at least one of $\Lambda_u(f_0) \notin \F_1$ and $\Lambda_u(f_4) \notin \F_1$ holds,
		so $T_u(f_0,f_4,v_0) \leq 1 + \frac{5}{6} + \frac{1}{2} = \frac{7}{3}$. By Claim \ref{FC-(5_2,3,4,4)}, 
		$f_i$ is not a $(5,3,4,4)$-face. 
		If $v_2$ has no $3$-neighbors,
		then for $i \in \{1,2\}$, we have $\Lambda_u(f_i) \notin \F_1$, and hence $T_u(f_i) \leq \frac{5}{6}$ by Rule~\ref{R5}. Thus $T_u(f_1,f_2,v_2) \leq \frac{5}{6} \times 2 + \frac{1}{3} = 2$,
		which implies that $c^*(u) \geq 5- \frac{7}{3} - 2 -\frac{1}{2} = \frac{1}{6} > 0$. Therefore we may assume that $v_2$ also has a $3$-neighbor. Without loss of generality, assume that $d(w_1)=3$ (implying that $d(v_1) \geq 5$),
		then $T_u(v_2,f_1) \leq 1+\frac{1}{2}=\frac{3}{2}$. 
		By Observation \ref{obs- almost 3/4}, we have $T_u(f_2) \leq \frac{3}{4}$.
		Then  $T_u(v_2,f_1,f_2) \leq \frac{9}{4}$.
		Similarly, we obtain $T_u(v_0,f_4,f_0) \leq \frac{9}{4}$.
		Therefore,
		$c^*(u) \geq 5 - \frac{9}{4} \times 2 - \frac{1}{2} = 0$. 
		
		\medskip
		\noindent
		{\bf Case 2.}  $u$ has three $3$-neighbors.
		\medskip   
		
		Without loss of generality, let $v_0$, $v_1$ and $v_{\ell}$ be the three $3$-neighbors of $u$, where $\ell \in \{2,3\}$. By Claim~\ref{FC-k_{k-2}-3_1}, none of $v_0$, $v_1$ and $v_{\ell}$ has a $3$-neighbor. Thus $T_u(v_0, v_1, v_{\ell})=1$ by Rule~\ref{R1}. By Lemma~\ref{lem-exist (5,3,4,3)-face}, we may assume that no $(5,3,4,3)$-face incident with $u$, so $T_u(f_0) \leq 1$ by Rule~\ref{R5}. By Claim~\ref{FC-5_3/4_24_1}, $v_i$ is not a $4_1$-vertex for any $i\in \{2,3,4\}\setminus \{\ell\}$. 
		
		If $\ell = 3$,  by Observation~\ref{obs- almost 3/4}, we have $T_u(f_i) \leq \frac{3}{4}$ for each $i \in [4]$. Thus $c^*(u) \geq 5- 1- 1 - \frac{3}{4} \times 4 = 0$. 
		If $\ell = 2$,  then $T_u(f_1) \leq 1$ by Rule~\ref{R5}.  Observation~\ref{obs- almost 3/4} implies $T_u(f_i) \leq \frac{3}{4}$ for any $i\in \{2,4\}$.
		Since $\Lambda_u(f_3)\notin\{(5,4_0,4_1,4_0),(5,4_0,4_2,4_0)\}$ by Claim~\ref{FC-(5_2,4,4_1,4)},
		$T_u(f_3) \leq \frac{1}{2}$ by Observation~\ref{obs- almost 1/2}. Thus $c^*(u) \geq 5- \frac{1}{3}\times 3 - 1 \times 2 - \frac{3}{4} \times 2 - \frac{1}{2} = 0$. 
	\end{proof}
	
	\subsection{Proof of Theorem~\ref{main-thm}}
	
	Since $G$ is a counterexample to Theorem~\ref{main-thm}, 
	every face $f$ satisfies $d(f)\ge 4$. 
	If $d(f)=4$, then $c^*(f)\ge 0$ by Lemma~\ref{lem-4-face}; 
	if $d(f)\ge 5$, then $c^*(f)=c(f)=2d(f)-10\ge 0$.
	
	By Claim~\ref{FC-min-degree}, $d(v)\ge 3$ for all $v\in V(G)$.
	If $d(v)=3$, then $c^*(v)\ge 0$ by Lemma~\ref{lem-3-vertex}.
	If $d(v)=4$, then $c^*(v)\ge 0$ by Lemma~\ref{lem-4-vertex}.
	If $d(v)=5$, then $c^*(v)\ge 0$ by Lemma~\ref{lem-5-vertex}.
	If $d(v)=6$, then $c^*(v)\ge 0$ by Lemma~\ref{lem-6-vertex}.
	If $d(v)\ge 7$, then $c^*(v)\ge 0$ by Lemma~\ref{lem-7-vertex}.
	
	Thus the final charge $c^*(z)$ is nonnegative for every 
	$z\in V(G)\cup F(G)$. 
	However, the total charge is preserved:
	\[
	-20=\sum_{z\in V(G)\cup F(G)} c(z)
	=\sum_{z\in V(G)\cup F(G)} c^*(z)\geq 0,
	\]
	a contradiction. 
	This completes the proof of Theorem~\ref{main-thm}.

\section{Acknowledgments}
	We thank Xinbo Xu for carefully reading our first draft and for pointing out an error in Claim~9. In this version, we have addressed the issue caused by this error. This work was initiated during the Workshop on Graph Theory and Combinatorics held in Kunming, China, in October 2025, which was supported by the 2025 National Tianyuan Program. We also thank Professors Xujin Chen, Jie Ma, and Guanghui Wang for organizing the workshop and for their kind invitation. The first author is supported by the National Natural Science Foundation of China (No.~12471325). The second author is supported by the National Natural Science Foundation for Young Scientists of China (Grant No.~12401472), and the Zhejiang Provincial Natural Science Foundation of China (Grant No.~LQN25A010011).

\end{document}